\documentclass[12pt]{article}

\usepackage{amsmath,amssymb,amsthm}
\usepackage{color,framed,accents,graphicx}

\normalsize

\newtheorem{theorem}{Theorem}[section]
\newtheorem{lemma}[theorem]{Lemma}

\newtheorem{corollary}[theorem]{Corollary}

\numberwithin{equation}{section}

\let\eps=\varepsilon
\def\rsets{{\mathcal{S}_r(m)}} 

\def\({\bigl(}   \def\){\bigr)}
\def\abs#1{\mathopen|#1\mathclose|} \let\card=\abs
\def\Abs#1{\bigl|#1\bigr|}

\def\evec{{\boldsymbol{e}}}
\def\jvec{{\boldsymbol{j}}}
\def\vvec{{\boldsymbol{v}}}
\def\xvec{{\boldsymbol{x}}}
\def\cvec{{\boldsymbol{c}}}
\def\yvec{{\boldsymbol{y}}}
\def\kvec{{\boldsymbol{k}}}

\def\thetavec{{\boldsymbol{\theta}}}

\def\Xvec{\boldsymbol{X}}

\def\nperp{n_{\scriptscriptstyle\perp}}

\newcommand{\medtilde}{\protect\accentset{\sim}}

\def\dfrac#1#2{\lower0.15ex\hbox{\large$\textstyle\frac{#1}{#2}$}}
\def\Dfrac#1#2{\raise0.05ex\hbox{\small$\displaystyle\frac{#1}{#2}$}}
\def\E{\operatorname{\mathbb{E}}}
\def\nicebreak{\vskip0pt plus50pt\penalty-300\vskip0pt plus-50pt }

\def\st{\mathrel{:}}

\def\dfrac#1#2{\lower0.15ex\hbox{\large$\frac{#1}{#2}$}}
\let\originalleft\left
\let\originalright\right
\renewcommand{\left}{\mathopen{}\mathclose\bgroup\originalleft}
\renewcommand{\right}{\aftergroup\egroup\originalright}

\definecolor{byzantine}{rgb}{0.74, 0.2, 0.64}

\definecolor{forestgreen}{rgb}{0.13, 0.55, 0.13}

\def\V{\operatorname{\mathbb{V\!}}}

 \let\Pr=\Prob

\let\le=\leqslant
\let\ge=\geqslant

\allowdisplaybreaks

\def\calR{\mathcal{R}}

\def\calB{\mathcal{B}}
\def\calT{\mathcal{T}}

\def\thetavec{\boldsymbol{\theta}}

\def\Var{\operatorname{Var}}
\def\Cov{\operatorname{Cov}}
\def\Reals{{\mathbb{R}}}
\def\Complexes{{\mathbb{C}}}

\def\Integers{{\mathbb{Z}}}

\def\norm#1{\mathopen\|#1\mathclose\|}

\def\trans{^{\mathrm{t}}}

\title{Enumeration of regular multipartite hypergraphs}
	\author{Mikhail Isaev\thanks
		{Supported by the Australian Research Council grant DP250101611.}\\
		\small School of Mathematics and Statistics\\[-0.8ex]
		\small University of New South Wales\\[-0.8ex]
		\small Sydney, NSW, Australia\\ 
		\small\tt  m.isaev@unsw.edu.au  
		\and
		Tam\'{a}s Makai\thanks
		{Supported by the Australian Research Council grant DP190100977.}\\
		\small Department of Mathematics \\[-0.8ex]
		\small LMU Munich\\[-0.8ex]
		\small Munich  80333, Germany\\
		\small \tt makai@math.lmu.de\\
		\and
		Brendan D. McKay\footnotemark[1]\\
		\small School of Computing\\[-0.8ex]
		\small Australian National University\\[-0.8ex]
		\small Canberra, ACT 2601, Australia\\
		\small\tt brendan.mckay@anu.edu.au}

\begin{document}
	\maketitle
	
	
	\begin{abstract}
		We determine the asymptotic number of regular multipartite
		hypergraphs, also known as multidimensional binary contingency
		tables, for all values of the parameters.
	\end{abstract}
	

	\section{Introduction}
	Let  $[n]:=\{1,\ldots, n\}$ be a vertex set partitioned into $r$ disjoint classes
	$V_1,\ldots,V_r$.
	We consider multipartite $r$-uniform  hypergraphs such that every edge has exactly one vertex in each class and there are no repeated edges.  
	We call such hypergraphs \textit{$(r,r)$-graphs}. Note that  $(2,2)$-graphs are just bipartite graphs. These objects  are also known as $r$-dimensional binary contingency tables.
	
	The \textit{degree} of a vertex is the number of edges that contain it. If every vertex has degree~$d$ then the hypergraph is called \textit{$d$-regular}, which implies that all the classes have the same size
	unless there are no edges. 
	We are interested in the number of labelled $d$-regular $(r,r)$-graphs with $n = mr$ vertices where every partition class contains exactly $m$ vertices. We denote this number by $H_r(d,m)$.
	
	In order to motivate our answer, we start with a non-rigorous argument.
	Consider an $(r,r)$-graph $G$ with classes
	of size $m$ and $md$ edges, created by choosing uniformly at random
	$md$ distinct edges out of the $m^r$ available. 
	Let $\calR$ be the event that $G$ is $d$-regular.
	Then
	\[
	H_r(d,m) = \binom{m^r}{md}\Pr(\calR).
	\]
	To estimate $\Pr(\calR)$, we first look at one class $V_i$, $i\in [r]$. The probability of the event $\calR_i$ that
	every vertex in $V_i$ has degree $d$ is
	\[
	\Pr(\calR_i):=\frac{\binom{m^{r-1}}{d}^m}{\binom{m^r}{md}},
	\]
	since
	$\binom{m^{r-1}}{d}$ is the number of ways to choose $d$ edges of an $(r,r)$-graph
	incident to one vertex in~$V_i$ and these
	choices are independent.
	If the events $\calR_i$ were also independent, we would have
	$\Pr(\calR)=\prod_{i=1}^r \Pr(\calR_i)$, providing the
	estimate
	\begin{equation}\label{eq:naive}
		\hat H_r(d,m) :=  \binom{m^r}{md} \prod_{i\in [r]} \Pr(\calR_i)= \frac{\binom{m^{r-1}}{d}^{rm}}{\binom{m^r}{md}^{r-1}}.
	\end{equation}
	Of course, the events $\calR_i$ are not independent,
	but comparing $\hat H_r(d,m)$ with the correct value
	$H_r(d,m)$ will be instructive.
	
	In the case of $(2,2)$-graphs, that is for $d$-regular bipartite graphs, we have as $n=2m\to\infty$ 
	that
	\[
	H_2(d,m) = (1+o(1))\, e^{-1/2} \hat H_2(d,m)
	\]
	except in the trivial cases $d=0$ and $d=m$.
	This follows from the work of
	McKay~\cite{McKay} (see also~\cite{GreenhillMcKay}),
	Canfield and McKay~\cite{CanfieldMcKay}, and
	Liebenau and Wormald~\cite{LiebenauWormald}, which
	together cover all densities.
	
	In this paper we show that, for $r\ge 3$, the estimate $ \hat H_r(d,m)$ is asymptotically correct.

	\begin{theorem}\label{thm:beautiful}
		Let $n=rm\to\infty$, where  $m=m(n)$ and $r=r(n)\ge 3$.  Then  for any $0\le d \le m^{r-1}$ we have
		\[ 
		H_r(d,m) = (1+o(1))\, \hat H_r(d,m).
		\]
	\end{theorem}
	
	For $r\ge 3$ we solve the problem in four ranges of parameters
	that together cover all possible values.
	In the sparse range we employ a combinatorial model
	introduced in 1972 by B\'ek\'essy, B\'ek\'essy and Koml\'os~\cite{Bekessy} and later developed under the name ``configurations'' by Bollob\'as and
	others~\cite{Bollobas1980,Wormald1999}.
	For $r\ge 4$ we show that this model produces a uniform random $d$-regular $(r,r)$-graph with probability $1-o(1)$. When $r=3$, this is insufficient to cover all of the sparse cases. However, the result can be extended to the whole sparse range by applying the method of switchings~\cite{McKay}. For the dense regime, we use the complex-analytic approach, relying on the machinery developed by Isaev and McKay \cite{Mother}.
	
	In Section~\ref{s:mproof} we introduce our key lemmas for both the sparse and the dense regimes and prove Theorem~\ref{thm:beautiful}. We then consider the sparse regime in Section~\ref{s:sparse} and the dense regime in Section~\ref{s:dense1}.
	
	\noindent\textbf{Notation. }
	Let $g(N)$ be a positive real function.
	For a complex function $f(N)$, the notation $f=O(g)$ means that $\abs{f(N)}\le Cg(N)$ for all positive $N$
	and some constant~$C$,
	while $f=o(g)$ means that $f(N)/g(N)\to 0$ as $N\to\infty$.
	For a real function $f(N)$, the
	notation $f=\Omega(g)$ means that $f(N)\ge Cg(N)$
	for all positive $N$ and some constant~$C>0$.
	We use $\log$ to denote the natural logarithm.
	
	\nicebreak
	\section{The four regimes  covering all possibilities}\label{s:mproof}
	
	We prove Theorem~\ref{thm:beautiful},
	considering the following four regimes separately.
	\begin{itemize}
		\item[(a)]  $d = o(m)$ for $r =3$;
		\item[(b)]  $rd^2 = o(m^{r-2})$ for $r\ge 3$;
		\item[(c)]  The complements of cases (a) and (b);
		i.e., replacing $d$ by $m^{r-1}-d$.
		\item[(d)]  
		$\min\{d, m^{r-1}-d\} = \Omega(r^{16}m)$.
	\end{itemize}
	
	\noindent
	Each of the following lemmas covers one of the regimes (a),(b),(d),
	while region (c) follows from  (a) and (b) as
	$H_r(m^{r-1}{-}d,m)=H_r(d,m)$ and
	$\hat H_r(m^{r-1}{-}d,m)=\hat H_r(d,m)$.
	
	\begin{lemma}\label{L:sparse2}
		If  $d = o(m)$ then   $H_3(d,m) = (1+o(1))\hat{H}_3(d,m) $.
	\end{lemma}
	
	\begin{lemma}\label{L:sparse1}
		If $r\ge 3$ and $rd^2 = o(m^{r-2})$ then  $H_r(d,m) = (1+o(1))\hat{H}_r(d,m) $.
	\end{lemma}
	
	\begin{lemma}\label{L:dense}
		If $r\ge 3$ and  $\min\{d, m^{r-1}-d\} = \Omega(r^{16}m)$, then $H_r(d,m) = (1+o(1))\hat{H}_r(d,m)$.
	\end{lemma}
	
	We prove Lemmas~\ref{L:sparse2} and \ref{L:sparse1} in Section~\ref{s:sparse} and Lemma~\ref{L:dense} in Section~\ref{s:dense1}.
	To see that the four regimes (a)--(d) together cover $0\le d\le m^{r-1}$, we employ the following short lemma. It is stated in a slightly more general form as it will
	also be useful in a few other places in the paper.
	
	\begin{lemma}\label{L:errors}
		Let $a,b,c,r_0$ be constants with $c>0$.
		Then, as $n\to\infty$ with $r\ge r_0$ and $m\ge 2$,
		\[
		r^a m^{-cr+b} = O(n^{b-cr_0}).
		\]
	\end{lemma}
	
	\begin{proof}
		Recall that $n=mr$.
		If $m=2$ then $r^a m^{-cr+b}=e^{-\Omega(n)}$ so the stated bound is immediate. 
		For $m\ge 3$, to maximise the function 
		$f_n(r) :=  r^{a} (n/r)^{-cr+b}$ over $r$, we  observe that
		the derivative
		\[
		\Dfrac{d}{dr}\log f_n(r) 
		= \dfrac{a-b}{r} - c\,(\log n - \log r - 1)
		= \Dfrac{a-b}{r} - c\,(\log m - 1)
		\]
		is negative for $r>\frac{a-b}{c\,(\log m-1)}$.
		Thus, the maximum occurs for $r=O(1)$, where the implicit constant in $O(1)$ depends only on $a,b,c$.  
		The result follows as for every $r\ge r_0$ satisfying $r=O(1)$ we have $r^a m^{-cr+b}=O(n^{b-cr_0})$.
	\end{proof}

	\begin{proof}[Proof of Theorem~\ref{thm:beautiful}]
		First note that the statement holds trivially when $m=1$. For the remainder of the proof assume that $m\ge 2$.
		
		For $r=3$ the result holds by Lemma~\ref{L:sparse2} for the regime 
		$d=o(m)$, and when $m^2-d=o(m)$, by taking the complement hypergraph. 
		On the other hand  Lemma~\ref{L:dense} proves the result for   $\min\{d,m^2-d\}=\Omega(m)$, covering the remaining possible degrees.
		
		For $r\ge 4$, Lemma~\ref{L:sparse1} applied to the hypergraph or its complement, 
		proves the result for the regimes $d=o(m^{(r-2)/2}r^{-1/2})$ and $m^{r-1}{-}d=o(m^{(r-2)/2}r^{-1/2})$.
		Furthermore Lemma~\ref{L:dense} proves the result for  $\min\{d,m^{r-1}-d\}=\Omega(r^{16} m)$.
		Observe that all possible degrees are covered if $m^{(r-2)/2}r^{-1/2}=\Omega(r^{16} m)$ or equivalently
		\[
		r^{33}m^{4-r}=O(1),
		\]
		which follows from Lemma~\ref{L:errors}, by taking $a=33$, $b=r_0=4$, and $c=1$.
	\end{proof}

	\nicebreak
	\section{Sparse regimes}\label{s:sparse}
	
	In this section we prove Lemma~\ref{L:sparse2} and Lemma~\ref{L:sparse1}. 
	First, we introduce the configuration model.
	Take $r$ disjoint classes $V_1,\ldots,V_r$ of $m$ vertices each,
	and an abstract set of $rmd$ elements that we will call \textit{spines}.
	Each spine is \textit{attached to} (i.e.\ associated with) a
	unique vertex such that each of the $rm$ vertices has $d$ attached spines.
	A \textit{spine set} is a set of $r$ spines, one from each class.
	A \textit{configuration} is an unordered partition of the $rmd$ spines into $md$ disjoint spine sets.
	Thus, there are $((md)!)^{r-1}$ possible configurations.
	Each configuration \textit{provides} a $d$-regular multi-hypergraph where each edge consists of the vertices to which the spines in a spine set of the configuration are attached.
	
	If $G$ is a simple $d$-regular $(r,r)$-graph, the number of configurations that
	provide~$G$ is $(d!)^{mr}$.  
	Therefore,
	\begin{equation}\label{eq:H-simple}
		H_r(d,m) = \frac{((md)!)^{r-1}}{(d!)^{mr}} P_r(d,m),
	\end{equation}
	where $P_r(d,m)$ is the probability that a uniform random configuration provides a \emph{simple} hypergraph, that is, no two spine sets give the same hyperedge.

	\subsection{Proof of Lemma~\ref{L:sparse1}}
	Observe that 
	\[
	\hat{H}_r(d,m)= 
	\frac{\binom{m^{r-1}}{d}^{rm}}{\binom{m^r}{md}^{r-1}} 
	= \frac{\left((m^{r-1})_d\right)^{rm}}{\left((m^r)_{md}\right)^{r-1}} \cdot \frac{((md)!)^{r-1}}{(d!)^{rm}}
	=    e^{ O(rd^2 m^{2-r}) }  \frac{((md)!)^{r-1}}{(d!)^{rm}}.
	\]
	For the last equality in the above, 
	we used  the bounds
	\[ 
	(m^{r-1})_d = m^{d(r-1)} e^{O(d^2 m^{1-r})} \quad 
	\text{and} \quad (m^r)_{md} = m^{rmd} e^{O(d^2 m^{2-r})},
	\]
	which follow from $(x-N)^N\le (x)_N\le x^N$ for $x>N$.
	Therefore, by the assumptions  and \eqref{eq:H-simple},  it is sufficient to show that 
	$P_r(d,m) = 1-o(1)$. This is obvious for $d=0,1$ so we can assume $d\ge 2$. 
	
	Consider any two spines $s_1$ and $s_2$ attached to the same vertex from $V_1$. The probability that two spine sets of a random configuration that contain $s_1$ and $s_2$ coincide is 
	\[
	\Dfrac{(md)^{r-1} (d-1)^{r-1} ((md-2)!)^{r-1} }{((md)!)^{r-1}} 
	= \Bigl(\Dfrac{d-1}{md-1}\Bigr)^{r-1}
	\le m^{1-r}.  
	\]
	Indeed, $(md)^{r-1}$ represents the number of choices for the spine set containing $s_1$. Then, we have $(d-1)^{r-1}$ ways to form the spine set containing $s_2$ which corresponds to the same edge. Finally, $((md-2)!)^{r-1}$ is the number of ways to complete the configuration.
	By the union bound over all possible choices for $s_1$ and $s_2$, 
	the lemma assumptions give
	\[
	1- P_r(d,m) \le  m \binom{d}{2}  m^{1-r}  = o(1),
	\]
	completing the proof.

	\subsection{The case $r=3$}
	
	For   $r=3$ we need to cover the range $d=o(m)$, in order to complement the range given by the complex-analytic approach detailed in the next section.  
	To do this, we use a switching argument to estimate the probability that a uniform random configuration provides a simple $(r,r)$-graph.
	Throughout this section, we can also assume that $d \ge 2$ since Lemma~\ref{L:sparse1} already covers $d = o(m^{1/2})$.
	
	For a given configuration, with a slight abuse of notation, we refer to the spine sets in the configuration as \textit{edges},
	thus distinguishing them from other spine sets. 
	The set of vertices used by a spine set is the set of vertices to which the spines are attached.
	An edge is \textit{simple}
	if no other edge uses the same vertices, and \textit{half of a double edge}
	if it is an edge and there is exactly one different edge which uses
	the same vertices. Two spine sets are 
	\textit{parallel} if they use the same vertices.
	
	Define 
	\[
	M:=\lfloor 8d^2m^{-1}+\log m\rfloor.
	\]
	
	\begin{lemma}\label{L:sparsebegin}
		Suppose $r=3$ and $2\le d=o(m)$.  Then, with 
		probability $1-o(1)$, a~uniform random configuration:
		\begin{itemize}\itemsep=0pt
			\item[(a)] does not contain three edges which are mutually parallel; and
			\item[(b)] contains at most $M$ 
			pairs of edges which are parallel.
		\end{itemize}
	\end{lemma}
	
	\begin{proof}
		Similarly to the proof of Lemma~\ref{L:sparse1}, we consider  any three spines $s_1$, $s_2$, $s_3$ attached to the same vertex from $V_1$. The probability that the corresponding three edges of a random configuration are parallel is 
		\[
		\Dfrac{(md)^{2} (d-1)^{2}(d-2)^2  ((md-3)!)^{2} }{((md)!)^{2}} =   \left(\Dfrac{(d-1)(d-2)}{(md-1)(md-2)}\right)^2 \le m^{-4}.
		\]
		Taking the union bound over all choices of the spines $s_1, s_2, s_3$, we get that the probability of having three parallel edges is 
		at most $m \binom{d}{3} \cdot m^{-4} = o(1)$, which proves part~(a).
		
		In order to prove \emph{(b)} it is sufficient to consider disjoint pairs (that is, every edge appears in at most one pair).
		For $i=1,\ldots, M+1$, let 
		$\{s_1^{i}, s_2^{i}\}$  be distinct pairs of spines,  
		such that each spine in a pair is attached to the same vertex of $V_1$.
		The probability that the     edges corresponding to the spine sets containing $s_1^{i}$  and $s_2^{i}$ of a random configuration coincide for all $i$ is at most
		\begin{align*}
			\frac{(m d)^{2(M+1)}d^{2(M+1)} ((md-2(M+1))!)^2}{((md)!)^2}&\le  m^{-2(M+1)} e^{O(M^2/md)} 
			= (1+o(1))  m^{-2(M+1)},
		\end{align*}
		where we used the assumption that $d=o(m)$ and the definition of $M$.
		The number of ways to  choose all pairs $\{s_1^{i},s_2^{i}\}_{i\in [M+1]}$
		is  at most  $\binom{dm}{M+1}d^{M+1}\le (d^2 m)^{M+1}/(M+1)!$. Applying the union bound implies that the probability of having $M+1$ pairs of two parallel edges is at most
		\[
		\frac{\bigl((1+o(1))m^{-1}d^2\bigr)^{M+1}}{(M+1)!}
		\le \biggl(\frac{e(1+o(1))d^2m^{-1}}{M+1}\biggr)^{\!M+1}
		\le ((1+o(1))e/8)^{\log m} = o(m^{-1}).
		\]
		proving part~(b).
	\end{proof}

	From now on, we only consider configurations which
	satisfy
	conditions (a) and~(b) of Lemma~\ref{L:sparsebegin}.
	Our task is to estimate the probability that the
	configuration provides no double edges.
	
	\begin{figure}[ht]
		\unitlength=1mm
		\centering
		{\small
			\begin{picture}(120,65)
				\put(21,0){\includegraphics[scale=0.5]{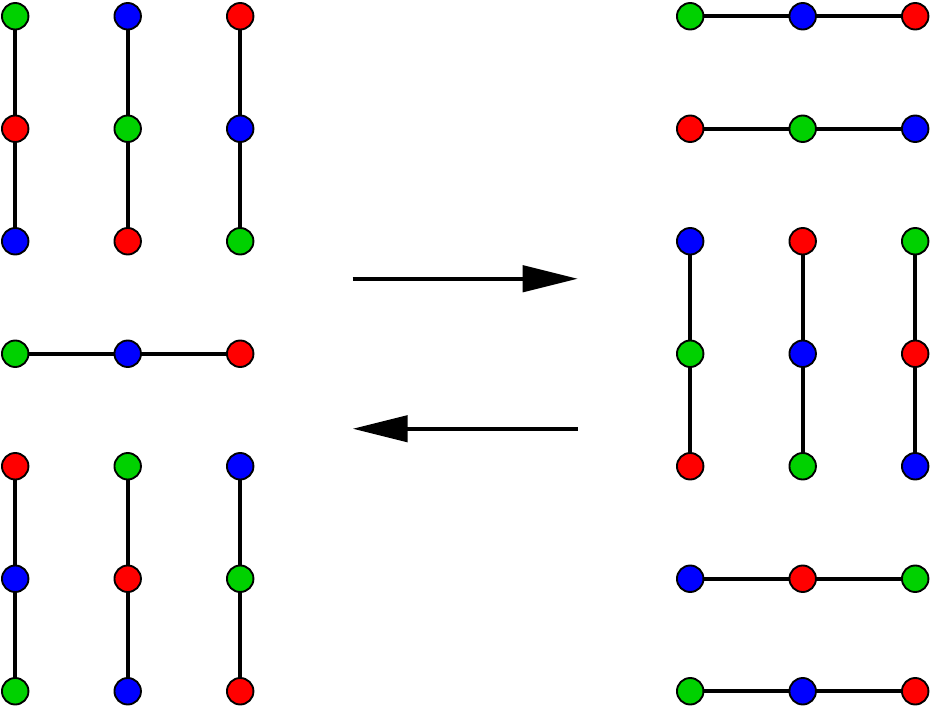}}
				\put(52.5,38.5){forward}\put(55,26){backward}
				\put(22,60.5){$a_6$} \put(31.5,60.5){$b_6$} \put(41,60.5){$c_6$}
				\put(23,50.5){$c_4$} \put(32.5,50.5){$a_4$} \put(42,50.5){$b_4$}
				\put(23,41){$b_2$} \put(32.5,41){$c_2$} \put(42,41){$a_2$}
				\put(22,32){$a$} \put(31.5,32){$b$} \put(41,32){$c$}
				\put(22,22.5){$c_1$} \put(31.5,22.5){$a_1$} \put(41,22.5){$b_1$}
				\put(23,12.5){$b_3$} \put(32.5,12.5){$c_3$} \put(42,12.5){$a_3$}
				\put(23,3){$a_5$} \put(32.5,3){$b_5$} \put(42,3){$c_5$}    
				\put(79,60.5){$a_6$} \put(88.5,60.5){$b_6$} \put(98,60.5){$c_6$}
				\put(79,50.8){$c_4$} \put(88.5,50.8){$a_4$} \put(98,50.8){$b_4$}
				\put(80.0,41){$b_2$} \put(89.5,41){$c_2$} \put(99.0,41){$a_2$}
				\put(80.5,32){$a$} \put(90,32){$b$} \put(99.5,32){$c$}
				\put(80.5,22.2){$c_1$} \put(90,22.2){$a_1$} \put(99.5,22.2){$b_1$}
				\put(79,12.7){$b_3$} \put(88.5,12.7){$c_3$} \put(98,12.7){$a_3$}
				\put(79,3.2){$a_5$} \put(88.5,3.2){$b_5$} \put(98,3.2){$c_5$}    
		\end{picture}}
		\caption{Switching operations for $r=3$.\label{fig:switchings}}
	\end{figure}
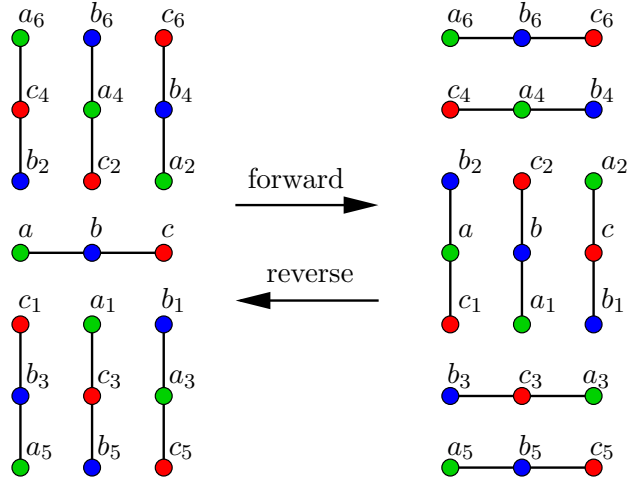
	
	We define \textit{forward} and \textit{backward switchings}
	involving 21 distinct spines labelled
	\[
	\{ a,b,c \} \cup \{ a_j,b_j,c_j \st 1\le j \le 6\}.
	\]
	In all cases, the letter in a spine name (and the colour in the figure) indicates which class $V_i$
	it belongs to.  All of these spines must be attached to distinct vertices,
	except that there may be coincidences within each of the sets
	\begin{equation}\label{allowed_coin}
		\{v(a),v(a_1),v(a_2)\}, \quad \{v(b),v(b_1),v(b_2)\}, \quad \text{and} \quad \{v(c),v(c_1),v(c_2)\},
	\end{equation}
	where $v(x)$ is the vertex to which spine~$x$ is attached.
	To describe our switching operations, we define two families of labelled spine sets (see Figure~\ref{fig:switchings}).
	\begin{align*}
		\mathcal{F}_1 &= \bigl\{\{a,b,c\}, \{a_1,b_5,c_3\},  \{a_2,b_4,c_6\},
		\{a_3,b_1,c_5\},  \{a_4,b_6,c_2\},\\
		&{\kern4em}\{a_5,b_3,c_1\}, \{a_6,b_2,c_4\}\bigr\} \\
		\mathcal{F}_2 &=  \bigl\{\{a,b_2,c_1\}, \{a_1,b,c_2\}, \{a_2,b_1,c\}\bigr\} 
		\cup \bigl\{ \{a_j,b_j,c_j\} \st 3\le j\le 6\bigr\}.
	\end{align*}
	
	Informally, a forward switching takes a configuration $C_1$ whose
	edges include $\mathcal{F}_1$, with $\{a,b,c\}$ being half
	of a double edge, and replaces the edges in $\mathcal{F}_1$
	by the edges in $\mathcal{F}_2$, thereby creating a new
	configuration $C_2$ without that double edge.
	A backward switching is the inverse operation that changes $C_2$ into $C_1$.
	We will restrict $C_1$ and $C_2$ so that neither the forward
	nor backward switchings create or destroy any double edge
	other than $\{a,b,c\}$ and its parallel spine set.
	
	In detail, the restrictions on $C_1$ are as follows.
	\begin{itemize}\itemsep=0pt
		\item[(F1)] $\{ a,b,c \} $ is half of a double edge, 
		while all the other edges in $\mathcal{F}_1$ are simple;
		\item[(F2)]  none of  the spine sets of $\mathcal{F}_2$ 
		are parallel 
		to edges of $C_1$;
		\item[(F3)] no two of the spine sets $\{a,b_2,c_1\}, \{a_1,b,c_2\}, \{a_2,b_1,c\}$ are parallel.
	\end{itemize}
	
	The restrictions on $C_2$ are as follows.
	\begin{itemize}\itemsep=0pt
		\item[(R1)] $\{a,b,c\}$ is not an edge of $C_2$ but
		is parallel to a simple edge of $C_2$,
		while all the edges in $\mathcal{F}_2$ are
		simple;
		\item[(R2)]  none of the spine sets in $\mathcal{F}_1$ 
		other than $\{a,b,c\}$ are parallel 
		to edges of $C_2$.
	\end{itemize}
	
	\smallskip
	We will need the following summation lemma from \cite[Cor.~4.5]{GreenhillMcKay}.
	
	\begin{lemma}[Greenhill, McKay, Wang \cite{GreenhillMcKay}]\label{sumlemma}
		Let $M \ge 2$ be an integer and, for $1\le i \le M$, let real numbers $A(i)$, $B(i)$
		be given such that $A(i) \ge 0$ and $1 - (i - 1)B(i) \ge 0$. Define
		\[ 
		A_1 := \min _{i\in [M]} A(i), \quad A_2 :=\max_{i\in [M]} A(i), \quad 
		C_1 := \min_{i\in [M]} A(i)B(i),  \quad C_2 := \max_{i\in [M]} A(i)B(i).
		\]
		Suppose that there exists
		$\hat{c}$ with $0 < \hat{c} < \frac13$
		such that $\max\{A/M, |C|\} \le \hat{c}$
		for all $A \in [A_1,A_2]$,
		$C \in [C_1,C_2]$. Define $n_0, \ldots , n_M$ recursively by $n_0 = 1$ and 
		\[
		n_i := \Dfrac 1i n_{i-1} A(i) (1 - (i -1)B(i)), \qquad \text{for $ i \in [M]$.}
		\]
		Then, the following bounds hold:
		\[
		\exp\bigl(A_1 - \dfrac12 A_1C_2\bigr)-(2e\hat{c})^M
		\le
		\sum_{i=0}^{M}n_i 
		\le
		\exp\bigl(A_2-\dfrac12A_2C_1+\dfrac12 A_2C_1^2\bigr)+ (2e\hat{c} )^M. 
		\]
	\end{lemma}
	Now we use Lemma~\ref{sumlemma} to estimate the probability $P_3(d,m)$ that the configuration model gives a simple hypergraph.
	\begin{lemma}\label{L:Tratio}
		If $2\le d=o(m)$, then
		$
		P_3(d,m) = (1+o(1)) \exp\Bigl( -\Dfrac{d^2}{2m} \Bigr).
		$
	\end{lemma}
	
	\begin{proof}
		For $0\le \ell\le M$, let $\calT(\ell)$ be the set of configurations that
		contain exactly $\ell$ pairs of edges which are parallel, but do not contain three mutually parallel edges.
		By Lemma~\ref{L:sparsebegin}, a random configuration
		belongs to $\bigcup_{\ell=0}^M \calT(\ell)$ with probability $1-o(1)$.
		Note that our assumption $d=o(m)$ implies that 
		\begin{equation}\label{M-bound}
			\eps_{\ell} := \frac{\ell+d}{md} \le \frac{M+d}{md}= o(1).  
		\end{equation}
		
		Consider a configuration $C_1$ in $\calT(\ell)$ for $1\le\ell\le M$.
		We claim that  
		\begin{equation}\label{forward-claim}
			\text{
				the number of available forward
				switchings is $(1+O(\eps_\ell))\, 2\ell m^6 d^6$.}
		\end{equation}
		Indeed, half of a double edge, $\{a,b,c\}$, can be chosen in $2\ell$ ways.
		Then we can choose 6 labelled simple edges, vertex-disjoint apart from the
		allowed coincidences described in~\eqref{allowed_coin}, in $(md - 2\ell - O(d))^6$ ways, thus satisfying (F1).
		
		The possibility forbidden by (F2) that
		an edge $e$ in $C_1$ is parallel to
		$\{a_j,b_j,c_j\}$ for some $3\le j\le 6$ 
		eliminates $O(\ell m^4d^7)$ cases, since we have $O(md)$ choices for the edge~$e$, 
		then there are $O(d^3)$ ways to select the three spines $a_j,b_j$ and $c_j$ (and thus also the three edges in $\mathcal{F}_1$ containing these spines),
		such that $\{a_j,b_j,c_j\}$ is parallel to $e$, 
		while the choice of the remaining edges in $\mathcal{F}_1$ is at most $2\ell(md)^3$.
		Similarly, the possibility that (F2) is violated because
		one of the spine sets
		$\{a,b_2,c_1\}$, $\{a_1,b,c_2\}$, or $\{a_2,b_1,c\}$ of $\mathcal{F}_2$
		is parallel 
		to an edge of $C_1$ eliminates $O(\ell m^4d^7)$ cases: 
		first select $\{a,b,c\}$ in $2\ell$ ways, then there are $O(d)$ ways to pick an edge $e$ in $C_1$ sharing a vertex with $a,b$ or $c$, at most $d^2$ ways to choose two edges of $\mathcal{F}_1$ each containing a spine which shares a vertex with one of the other two spines in $e$, 
		and at most $(md)^4$ ways to pick the remaining edges for $\mathcal{F}_1$.
		
		Finally, the number of ways to violate (F3) is $O(\ell m^3d^6)$.
		For example, suppose $\{a,b_2,c_1\}$ and $\{a_1, b, c_2\}$ are
		parallel. 
		Choose $\{a,b,c\}$ in $2\ell$ ways, then choose the three
		edges of $C_1$ containing $c_2$, $a_2$ or $b_1$ in at most $m^3d^3$
		ways. Then the number of ways to pick the remaining edges of $\mathcal{F}_1$ is at most $d^3$ because our assumption implies that $v(b_2)=v(b)$, $v(c_1)=v(c_2)$ and $v(a_1)=v(a)$.
		
		Using~\eqref{M-bound} and $d=o(m)$, we find that the number of available forward switchings is 
		\[
		2\ell\bigl( (md-2\ell-O(d))^6 - O(m^4 d^7) - O(m^3 d^6)\bigr) = 
		(1+O(\eps_\ell))\,2\ell m^6 d^6,
		\]
		which proves claim \eqref{forward-claim}.

		Next, consider a configuration $C_2$ in $\calT(\ell-1)$ for $1\le\ell\le M$.
		We claim that  
		\begin{equation}\label{backward-claim}
			\text{
				the number of available backward
				switchings is $(1+O(\eps_\ell))\, m^5 d^5 (d-1)^3$.}
		\end{equation}
		Indeed, we can choose a simple edge $\{a',b',c'\}$ in $md-2\ell+2$ ways,
		and then choose one additional spine $a,b,c$ on each of $v(a'),v(b'),v(c')$, respectively,
		in $(d-1)^3$ ways. 
		This identifies three edges of $C_2$ each containing one of $a,b,c$, recalling that their other vertices need not be unique.
		The possibility that they are not all simple eliminates
		$O(\ell (d-1)^3)$ cases: for example, if $\{a,b_2,c_1\}$ is not
		simple, there are at most $d-1$ choices for each of $a',b,c$, which identifies the other two edges.
		This gives $(md-O(\ell))(d-1)^3$ choices for
		$\{a,b_2,c_1\}, \{a_1,b,c_2\}, \{a_2,b_1,c\}$,
		following which we have $(md-2\ell+2-O(d))^4$ ways
		to choose the 4 simple edges $\{a_j,b_j,c_j\}$ for $3\le j\le 6$
		such that all 7 edges satisfy (R1) but not necessarily (R2).
		
		From this count we must subtract the violations of (R2).
		For example consider the case of an edge $e=\{a'',b'',c''\}$
		of $C_2$ being parallel to $\{a_1,b_5,c_3\}$.
		After choosing $\{a,b_2,c_1\}$, $\{a_1,b,c_2\}$
		and $\{a_2,b_1,c\}$
		in $O(md^4)$ ways as above, we can choose the spine $a''$ with
		$v(a'')=v(a_1)$ in $O(d)$ ways. This determines $b''$ and $c''$ 
		and we can choose $b_5$ with $v(b_5)=v(b'')$ and
		$c_3$ with $v(c_3)=v(c'')$, in $O(d^2)$ ways, thus
		determining the edges $\{a_3,b_3,c_3\}$ and $\{a_5,b_5,c_5\}$.
		Then the remaining two edges can be chosen in at most $(md)^2$ ways.
		Thus the allowance for (R2) is $O(m^3d^9)$.
		
		Using~\eqref{M-bound} and $d=o(m)$,
		we find that the number of backward switchings is
		\[ 
		(md-O(\ell))(d-1)^3 (md-2\ell+2-O(d))^4 - O(m^3d^9) =
		(1+O(\eps_\ell))\, m^5 d^5 (d-1)^3,
		\]
		which proves claim \eqref{backward-claim}.

		Consequently, by a simple double counting argument for the number of forward/backward switchings between
		$\calT(\ell)$ and $\calT(\ell-1)$,  we deduce that
		\[
		\frac{\card{\calT(\ell)}}{\card{\calT(\ell-1)}} = (1+ O(\eps_\ell))\Dfrac{(d-1)^3}{2\ell dm}
		= \bigl(1+ O(m^{-1}+ (\ell-1)(md)^{-1})\bigr)\Dfrac{(d-1)^3}{2\ell dm}.
		\]
		Moreover, it follows from the analysis that these
		estimates are uniform over $1\le\ell\le M$.
		Since $d= o(m)$, 
		to apply Lemma~\ref{sumlemma} for $n_\ell := \dfrac{\card{\calT(\ell)}}{\card{\calT(0)}}$, we can take
		\[
		A(\ell) = (1 + O(m^{-1}))\Dfrac{(d-1)^3}{2dm}\le
		(1+ o(1)) \Dfrac{d^2}{ 2 m } 
		\quad \text{and} \quad 
		B(\ell) = O((md)^{-1}).
		\]
		Recalling that $M = \lfloor 8d^2m^{-1}+\log m\rfloor \ge 8d^2 m^{-1}$, we obtain
		\[
		A(\ell)/M \le \dfrac18 
		\qquad
		\text{and} \qquad  
		C_1,C_2 = O\Bigl(\Dfrac{d}{m^2}\Bigr)
		= o(1) \le \dfrac{1}{8}.
		\] 
		Therefore, we can take $\hat{c} = \frac{1}{8}$. Applying Lemma~\ref{sumlemma},
		we obtain 
		\begin{align*}
			\Dfrac{\card{\calT(0)}}{\sum_{\ell=0}^M\, \card{\calT(\ell)}}
			= \biggl( \sum_{\ell=0}^M n_\ell\biggr)^{\!-1}
			&= \biggl(\exp\Bigl( \Dfrac{(d-1)^3}{2dm} + o(1)\Bigr) + O((e/4)^{M})\biggr)^{\!-1}  \\ 
			&=
			\exp\Bigl( -\Dfrac{(d-1)^3}{2dm} + o(1)\Bigr),
		\end{align*}
		since $M\to\infty$.
	\end{proof}
	
	Now we can compare $H_3(d,m)$
	with the estimate $\hat{H}_3(d,m)$ given in \eqref{eq:naive}.
	
	\begin{proof}[Proof of Lemma~\ref{L:sparse2}]
		Combining \eqref{eq:H-simple} and Lemma~\ref{L:Tratio}, we have 
		\[
		H_3(d,m) = \Dfrac{((md)!)^{2}}{(d!)^{3m}} 
		\exp\left(-\Dfrac{d^2}{2m} +o(1) \right).
		\]
		To estimate $\hat{H}_3(d,m)$ we argue similarly to Lemma~\ref{L:sparse1}. First, using $d= o(m)$, observe that  
		\begin{align*}
			(m^{2})_d = m^{2d }  \exp\left( -\Dfrac{d(d-1)
			}{2m^2} + O\left(d^{3}m^{-4}\right) \right) 
			=m^{2d }  \exp\left( -\Dfrac{d^2}{2m^2} + o (m^{-1})\right) \\ 
			(m^3)_{md} = m^{3md} \exp\left(-\Dfrac{d^2}{2m}  + O(dm^{-2}+d^3m^{-3})\right)=
			m^{3md} \exp\left(-\Dfrac{d^2}{2m}  + o(1)\right).
		\end{align*}
		We obtain
		\begin{align*}
			\hat{H}_3(d,m) &= 
			\Dfrac{\binom{m^{2}}{d}^{3m}}{\binom{m^3}{md}^{2}} 
			= \Dfrac{\left((m^{2})_d\right)^{3m}}{\left((m^3)_{md}\right)^{2}} \cdot \Dfrac{((md)!)^{2}}{(d!)^{3m}}
			=     \frac{((md)!)^{2}}{(d!)^{3m}}
			\exp\Bigl(-\Dfrac{d^2}{2m} +o(1) \Bigr) \\
			&= (1+o(1))H_3(d,m),
		\end{align*}
		as required.
	\end{proof}

	\section{Dense range}\label{s:dense1}

	Throughout this section we will assume that the
	partition classes  of the vertex set $[n]$ are
	\[
	V_t := \{(t-1)m+1, \ldots ,tm\}, \qquad \text{for $t\in [r]$.}
	\] 
	Let $\rsets$ denote the set of 
	all possible edges of $(r,r)$-graphs.
	Clearly, we have  
	\[
	n =mr \qquad \text{and} \qquad |\rsets| = m^r.
	\]
	Let 
	\[
	\lambda = \Dfrac{d}{m^{r-1}} \qquad  \text{ and } \qquad \varLambda = \lambda(1-\lambda).
	\]
	
	In this section we consider the dense range
	defined by $\varLambda=\Omega(r^{16}m^{2-r})$, which
	is equivalent to $\min\{d,m^{r-1}{-}d\} =\Omega(r^{16}m)$, using the complex-analytic approach.
	We establish a generating function for $(r,r)$-graphs by degrees, then extract the required coefficient via Fourier inversion
	and perform asymptotic analysis on the resulting multidimensional integrals.
	
	We will use the vector $p$-norm defined by 
	\[
	\norm{(x_1,\ldots,x_N)\trans}_p = (\abs{x_1}^p + \cdots + \abs{x_N}^p)^{1/p},
	\]
	together with its limit
	\[
	\norm{(x_1,\ldots,x_N)\trans}_\infty = \max\nolimits_{i=1}^N \abs{x_i}
	\]
	and the corresponding induced matrix norms.
	Note that $\norm{\cdot}_2$ is the spectral norm.
	
	
	\subsection{An integral for the number of $(r,r)$-graphs}

	The  generating function for $(r,r)$-graphs by degree sequence is 
	\[
	\prod_{e \in \rsets} \biggl(1+\prod_{j\in e} x_j\biggr).
	\]
	Using Cauchy's coefficient formula, the number of $d$-regular $(r,r)$-graphs is
	\begin{align*} H_r(d,m) &= [x_1^{d}\cdots x_n^{d}] 
		\prod_{e \in \rsets } \biggl(1 + \prod_{j\in e} x_j\biggr)\\
		&= \frac{1}{(2\pi i)^{rm}}\, \oint\cdots\oint\,   
		\Dfrac{\prod_{e \in \rsets }\bigl(1 + \prod_{j\in e} x_j\bigr)}{ \prod_{j\in [rm]} x_j^{d+1}}\, d\xvec.
	\end{align*}
	
	Considering the contours $x_j=  \left(\dfrac{\lambda}{1-\lambda}\right)^{1/r} e^{i \theta_j}$ for $j\in [n]$, we obtain
	\begin{align}
		H_r(d,m) &= (2\pi)^{-n} \, \left(\dfrac{\lambda}{1-\lambda}\right)^{-dm}\  
		\int_{-\pi}^\pi \!\cdots\! \int_{-\pi}^\pi \Dfrac{\prod_{e \in \rsets}
			\bigl(1 + \frac{\lambda}{1-\lambda} \prod_{j\in e}   e^{ i\theta_j}\bigr)}
		{\exp\bigl(i\sum_{j\in [rm]} d\theta_j\bigr)} \, d\thetavec \nonumber\\
		&= \Dfrac{(2\pi)^{-n}} {\left(\lambda^{\lambda} (1-\lambda)^{1-\lambda} \right)^{m^r}}  
		\int_{U_{n}(\pi)} \Dfrac{\prod_{e\in \rsets}
			\bigl(1 + \lambda \bigl(e^{ i\sum_{j\in e}\theta_j}-1\bigr)\bigr)}
		{\exp\bigl(id \sum_{j\in [rm]} \theta_j\bigr)} \, d\thetavec, \label{H-int1}
	\end{align}
	where
	\[
	U_{n}(\rho) = \{ \thetavec\in \Reals^n \st \norm{\thetavec}_\infty\le \rho\}.
	\]

	Denote the integrand of \eqref{H-int1} by $F(\thetavec)$, that is,
	\[
	F(\thetavec):= \Dfrac{\prod_{e \in \rsets}
		\bigl(1 + \lambda \bigl(e^{ i\sum_{j\in e}\theta_j}-1\bigr)\bigr)}
	{\exp\bigl(i d\sum_{j\in [n]} \theta_j\bigr)}.
	\]
	Since $\Abs{1+c(e^{ix}-1)}^2=1+2c(1-c)(\cos x-1)$
	for real $c,x$,
	the absolute value of $F(\thetavec)$ is
	\[
	\prod_{e \in\rsets} 
	\sqrt{1+2\varLambda\Bigl(\cos \Bigl(\textstyle\sum_{j\in e}\theta_j\Bigr)-1\Bigr)}\; .
	\]
	From this we can see that $\abs{F(\thetavec)}\le 1$, with equality
	if and only if  $\sum_{j\in e}\theta_j=0\pmod{2\pi}$ for each $e\in\rsets$.
	This is equivalent to the existence of some constants $c_1,\ldots,c_r$ whose sum is
	$0$ modulo $2\pi$ such that 
	$\theta_j=c_{t}$ for all $t \in [r]$ and $j \in V_t$. 
	This motivates us to introduce the transformation  $\Phi_{\cvec}(\thetavec) = (\varphi_1,\ldots,\varphi_n)\trans \in (-\pi,\pi]^n$, for arbitrary $\cvec=(c_1,\ldots,c_r)\trans \in \Reals^r$ defined by 
	\[
	\varphi_j = \theta_j + c_t \pmod{2\pi} \qquad  \text{ for all } t\in [r], \ j\in V_t.
	\]
	Note that if $c_1 + \cdots +c_r  = 0$  modulo $2\pi$ then
	\begin{equation}\label{F-symmetry}
		F(\Phi_{\cvec}(\thetavec)) = F(\thetavec).
	\end{equation}
	Using this symmetry
	we can reduce the integral in \eqref{H-int1} to an integral over a subset  $\calB$ of dimension $n-r+1$, where 
	\[
	\calB = \{\thetavec \in U_n(\pi) \st \theta_{2m} = \theta_{3m}=\cdots = \theta_{rm} =0\}.
	\]
	\begin{lemma}\label{L:Qlemma}
		We have
		\[
		H_r(d,m)   =  
		\frac{(2\pi)^{r-n-1}} {\left(\lambda^{\lambda} (1-\lambda)^{1-\lambda} \right)^{m^r}}     \int_{\calB} F(\thetavec)\, d\thetavec.
		\]
	\end{lemma}
	
	\begin{proof}
		Define $\cvec = \cvec( \theta_{2m}, \ldots,  \theta_{rm})$ by
		\[
		c_1 = \sum_{j=2}^r \theta_{jm} \quad \text{ and } \quad c_j = -\theta_{jm} \quad 
		\text{for all $j\ge 2$}.
		\]
		Then $\Phi_{\cvec}(\thetavec)\in \calB$. Using \eqref{F-symmetry}, we get $F(\Phi_{\cvec} (\thetavec)) = F(\thetavec)$.
		Integrating over 
		$\theta_{2m}, \ldots,  \theta_{rm}$ and then over the remaining coordinates separately, we get
		\[
		\int_{U_{n}(\pi)} F(\thetavec) \,d\thetavec =  
		(2\pi)^{r-1} \int_{\calB} F(\thetavec)\, d\thetavec.
		\]
		The result follows from \eqref{H-int1}.
	\end{proof}
	
	The advantage of considering the integral over the set $\calB$ of smaller dimension is that there is only one point for which
	$\abs{F(\thetavec)} =1$, namely the zero vector.
	Let 
	\begin{equation}\label{def:rho0}
		\rho_0:=\rho_0(\lambda)= \frac{r^{5/2}\log{n}} {(\varLambda m^{r-1})^{1/2}}.
	\end{equation}
	In Section~\ref{S:outside}, we show that 
	the contribution of the region
	$\calB \setminus U_{n}(\rho_0)$ to the integral 
	$\int_{\cal B} F(\thetavec) \,d\thetavec$
	is asymptotically negligible.

	A theoretical framework for asymptotically estimating integrals over truncated multivariate Gaussian distributions, including cases where the covariance matrix is not of full rank, 
	was developed by Isaev and McKay in \cite[Section 4]{Mother}. 
	Using the tools from \cite{Mother}, we estimate the integral of $F(\thetavec)$ over $\calB\cap U_n(\rho_0)$ in Section~\ref{S:inside}.

	
	\subsection{Outside the main region}\label{S:outside}
	
	In this section, we  estimate the contribution of $|F(\thetavec)|$ on $\calB\setminus U_n(\rho_0)$, where 
	$\rho_0$ is defined in \eqref{def:rho0}.
	Recall also that $\varLambda = \lambda(1-\lambda)$.
	Let 
	\[
	|x|_{2\pi} = \min_{k\in \Integers} |x-2\pi k|.
	\]
	Note that $|x-y|_{2\pi}$ is the circular distance between $x,y \in [-\pi,\pi]$. In particular, it never exceeds $\pi$ and satisfies the triangle inequality. 
	Our arguments rely on the next two lemmas. 
	\begin{lemma}\label{l:bound-factor}
		For any $x \in \Reals$ and $\lambda \in [0,1]$, we have
		\[
		\abs{1+  \lambda (e^{ix}-1)} = \sqrt{1-2\varLambda(1-\cos x)} \le 
		\exp\left(-\dfrac{\varLambda}{2}\left(1-  \dfrac{|x|^2_{2\pi}}{12}\right)  |x|^2_{2\pi}\right).
		\]
	\end{lemma}
	\begin{proof}
		Since $\cos x$ and the final bound are both periodic mod
		$2\pi$, we need only to check the interval $[-\pi,\pi]$.
		Noting that $\log\bigl(1-2\varLambda(1-\cos x)\bigr)
		\le -2\varLambda(1-\cos x)$, the claim follows
		from the elementary observation that $\cos x\le 1 - x^2/2 + x^4/24$ for $x\in[-\pi,\pi]$.
	\end{proof}
	
	\begin{lemma}\label{l:abs_inside}
		For $\varLambda=\Omega(r^{16} m^{2-r})$, we have
		\[ 
		\int_{\calB \cap U_n(\rho_0)} 
		|F(\thetavec)|\,d\thetavec
		\le    m^{(r-1)/2}  \left(\Dfrac{2\pi }{\varLambda m^{r-1}}\right)^{(n-r+1)/2} \exp\left( O\left(\Dfrac{r^2}{\varLambda m^{r-2}} \right)+O(n^{-2/5})\right).
		\] 
	\end{lemma}
	
	
	We will prove Lemma~\ref{l:abs_inside} and
	estimate $\int_{\calB\cap U_n(\rho_0)} F(\thetavec)\,d\thetavec$
	in Section~\ref{S:inside}.
	
	\medskip
	
	First, we consider the region $\calB_{1}$, where the components $\theta_j$ are widely spread apart within at least one class $V_t$. 
	Define
	\[
	\calB_{1}:=\Bigl\{\thetavec\in \calB \st  
	\text{$\exists t\in [r]$ such that } 
	\Bigl|\{\theta_j\}_{j \in V_t}  \setminus \bigl[\theta_k \pm \dfrac{\rho_0}{4r}\bigr]_{2\pi}\Bigr|>m/r
	\text{  for  all $k\in V_t$}\Bigr\},
	\]
	where  $[x\pm \rho]_{2\pi}$ is   the set of points at circular distance at most $\rho$ from $x$
	on  $[-\pi,\pi]$:
	\[
	[x\pm \rho]_{2\pi} = \{y \in [-\pi,\pi]\st |x-y|_{2\pi} \le \rho\}.
	\]

	\begin{lemma}\label{L:B1}
		If $\varLambda=\Omega(r^{16}m^{2-r})$, we have
		\[
		\int_{\calB_1}|F(\thetavec)|\, d\thetavec  \le e^{-\omega(rn\log n)}m^{(r-1)/2}  \left(\frac{2\pi }{\varLambda m^{r-1}}\right)^{(n-r+1)/2}.
		\]
	\end{lemma}
	
	\begin{proof}
		Using Lemma~\ref{l:bound-factor} for each factor in $F(\thetavec)$,  we get
		\begin{equation}\label{F-2pi-bound}
			|F(\thetavec)|\le \exp\biggl(-\Omega\biggl(\varLambda \sum_{e \in \rsets} \Bigl|\sum_{j\in e} \theta_j\Bigr|_{2\pi}^2 \biggr)\biggr).
		\end{equation}
		Let $V_t$ be a  class such that  $\Bigl|\{\theta_j\}_{j \in V_t}  \setminus \bigl[\theta_k \pm \dfrac{\rho_0}{4r}\bigr]_{2\pi}\Bigr|>m/r$
		for  all $k\in V_t$.
		Let $a,b\in V_t$.
		For any $e_1, e_2 \in \rsets$ such that $e_1 \mathbin{\triangle} e_2 =\{a,b\}$, we have 
		\[
		\biggl|\,\sum_{j\in e_1} \theta_j - \sum_{j\in e_2} \theta_j\biggr|_{2\pi}\!\!=|\theta_a-\theta_b|_{2\pi}.
		\]
		This implies that 
		\begin{equation}\label{eq:e12}
			\biggl|\sum_{j\in e_1} \theta_j\biggr|_{2\pi}^2 + 
			\biggl|\sum_{j\in e_2} \theta_j\biggr|_{2\pi}^2 
			\ge \Dfrac{1}{2}\Biggl(\biggl|\sum_{j\in e_1} \theta_j\biggr|_{2\pi} +
			\biggl|\sum_{j\in e_2} \theta_j\biggr|_{2\pi} 
			\Biggr)^{\!2}
			\ge \frac{|\theta_a-\theta_b|_{2\pi}^2}{2}.
		\end{equation}
		Summing  \eqref{eq:e12} over all choices of such pairs of edges $e_1,e_2$  and $a,b \in V_t$, we get 
		\[
		m \sum_{e \in \rsets} \Bigl|\sum_{j\in e} \theta_j\Bigr|_{2\pi}^2 \ge \frac{m^{r-1} }{2}\sum_{a,b\in V_t}|\theta_a-\theta_b|_{2\pi}^2.
		\]
		By the choice of $V_t$, for any $a\in V_t$, 
		there are  at least $m/r$ components $b\in V_t$ such that 
		$|\theta_a-\theta_b|_{2\pi} \ge \dfrac{\rho_0}{4r}$, implying
		\[
		\sum_{a,b\in V_t}|\theta_a-\theta_b|_{2\pi}^2 \ge m\cdot \min_{a\in V_t} \sum_{b\in V_t} |\theta_a-\theta_b|_{2\pi}^2\ge \Dfrac{m^2}{r} \left(\Dfrac{\rho_0}{4r}\right)^2.
		\]
		Using  the above bounds in \eqref{F-2pi-bound}, we find that 
		\[ 
		|F(\thetavec)| \le \exp\left(-\Omega\left(\varLambda m^{r }   \rho_0^2/r^3\right)\right).
		\]
		Note that the volume of $\calB_1$ cannot exceed the volume of $\calB$, which is $(2\pi)^{n-r+1}$.
		Then, recalling  from   \eqref{def:rho0}   the definition  of $\rho_0$ and that $n=mr$, we get
		\[
		\int_{\calB_1}\abs{F(\thetavec)}\, d\thetavec\le (2\pi)^{n-r+1} \exp\left(-\Omega\left(\varLambda m^{r}  \rho_0^2/r^3\right)\right) \le e^{-\omega(rn \log n)}.
		\]
		Observing also that
		\[
		m^{(r-1)/2} \left(\frac{2\pi }{\varLambda m^{r-1}}\right)^{(n-r+1)/2} = e^{O(rn \log n)},
		\]
		the claimed bound follows.
	\end{proof}
	
	The next region $\calB_2$ consists of all $\thetavec\in \calB$
	where within each class most of the components are close, and in addition after selecting a representative component for each of these close groups, the sum of these representatives is far from 0. This in turn implies that after summing one element of each of these close groups their sum will also be far from 0.
	Formally, there exist $(k_1,\ldots,k_r) \in V_1\times \cdots \times V_r$ satisfying the following two conditions:
	\[  \Abs{\theta_{k_1} + \cdots +\theta_{k_r}}_{2\pi} \ge \Dfrac{\rho_0}{3} \quad \text{and} \quad 
	\Bigl|\{\theta_j\}_{j \in V_t}  \cap \bigl[\theta_{k_t} \pm \dfrac{\rho_0}{4r}\bigr]_{2\pi}\Bigr|\ge m - m/r \text{ for all $t \in [r]$.}
	\]
	\begin{lemma}\label{L:B2}
		If $\varLambda=\Omega(r^{16}m^{2-r})$, we have
		\[
		\int_{\calB_2}\abs{F(\thetavec)}\, d\thetavec  \le e^{-\omega(rn\log n)}m^{(r-1)/2}  \left(\Dfrac{2\pi }{\varLambda m^{r-1}}\right)^{(n-r+1)/2}.
		\]
	\end{lemma}
	\begin{proof}
		Let  $\thetavec \in \calB_2$.  Consider any edge $e \in \rsets$
		such that, for all $t\in [r]$, and $j \in e\cap V_t$ we have 
		$|\theta_j-\theta_{k_t}|_{2\pi} \le \dfrac{\rho_0}{4r}$. Then, by the definition of $\calB_2$, we have
		\[
		\biggl|\sum_{j \in e} \theta_j\biggr|_{2\pi} 
		\ge    \biggl|\sum_{t\in [r]} \theta_{k_t}\biggr|_{2\pi}  - 
		\Dfrac{\rho_0}{4} \ge \Dfrac{\rho_0}{12}.
		\]
		The number of such edges is at least $(m-m/r)^r \ge m^r/4$.
		Using Lemma~\ref{l:bound-factor} for each corresponding factor, we obtain
		\[ 
		|F(\thetavec)| \le \exp\left(-\Omega\left(\varLambda m^{r}  \rho_0^2\right)\right).
		\]
		We complete the proof by repeating the arguments of Lemma~\ref{L:B1}.
	\end{proof}
	
	Next, for  $\kvec = (k_1,\ldots,k_r) \in  V_1\times \cdots \times V_r$ and $S\subseteq [n]\setminus\{k_1,\ldots,k_r\}$ with $|S\cap V_1|=\cdots=|S\cap V_r|$
	consider the regions $\calB_{\kvec,S}$ of $\thetavec\in \calB$ such that
	the following hold:
	\begin{itemize}
		\item $|\theta_{k_1} + \cdots +\theta_{k_r}|_{2\pi} \le \rho_0/3$;
		\item $|\{\theta_j\}_{j \in V_t}  \cap [\theta_{k_t} \pm \rho_0/(4r)]_{2\pi}|\ge m - m/r$ for all $t\in [r]$;
		\item $S$ contains $\bigcup_{t\in [r]}\{j\in V_t\st |\theta_j -\theta_{k_t}|_{2\pi}> \rho_0/(2r)\}$;
		\item there is $t\in[r]$ such that $\{j\in V_t\st |\theta_j -\theta_{k_t}|_{2\pi}> \rho_0/(2r)\}=S\cap V_t$.
	\end{itemize}

	\begin{lemma}\label{L:BS} If $\varLambda=\Omega(r^{16}m^{2-r})$, we have uniformly for every $\kvec \in  V_1\times \cdots \times V_r$ and $S\subseteq [n]\setminus \{k_1,\ldots,k_r\}$ with $|S|\le m$ that
		\[
		\int_{\calB_{\kvec,S}} \abs{F(\thetavec)}\, d\thetavec\le 
		e^{-\omega(|S|  \log n)} m^{(r-1)/2}  \left(\frac{2\pi }{\varLambda m^{r-1}}\right)^{(n-r+1)/2}.
		\]
	\end{lemma}
	
	\begin{proof}
		Due to the symmetry \eqref{F-symmetry} 
		and relabeling vertices, we can assume that $k_1 = m$, $k_2 = 2m$, $\ldots{}$, $k_r = rm$.
		
		The bounds indicated in this proof by the asymptotic notations
		$O(\,)$, $\Omega(\,)$ and $\omega(\,)$ can be chosen independently of~$S$.
		
		By definition, for $\thetavec \in \calB_{\kvec,S} \subset \calB$, we  have 
		\begin{equation}\label{eq:ourcase}
			\theta_{k_2} = \cdots =\theta_{k_r} = 0 \quad  \text{ and } \quad |\theta_{k_1}|_{2\pi} \le \dfrac{\rho_0}{3}.
		\end{equation}
		By definition there exists some $t\in[r]$ such that 
		\[
		\{j\in V_t\st |\theta_j -\theta_{k_t}|_{2\pi}> \rho_0/(2r)\}=V_t\cap S,
		\]
		and note that this set has size $|S|/r$.
		
		Consider any $e_1,e_2 \in \rsets$ such that  $e_1 \mathbin{\triangle} e_2 =\{a,b\}$, where 
		$a\in S \cap V_t$, and $b\in V_t$ satisfies $|\theta_b-\theta_{k_t}|_{2\pi} \le \dfrac{\rho_0}{4r}$.
		Let
		\[
		\Sigma_1 := \sum_{\substack{e \in  \rsets \\ e \cap S = \emptyset}} 
		\biggl|\sum_{j \in e} \theta_j\biggr|_{2\pi}^2 \qquad \text{and} \qquad 
		\Sigma_2 := \sum_{\substack{e \in  \rsets \\ e \cap S \neq \emptyset}} 
		\biggl|\sum_{j \in e} \theta_j\biggr|_{2\pi}^2.
		\]
		Arguing similarly to Lemma~\ref{L:B1} and using \eqref{eq:e12}, we find that  
		\begin{equation}\label{ineq:sigma12}
			\begin{aligned}
				|S\cap V_t| \cdot \Sigma_1 + m \cdot \Sigma_2 
				& \ge  m^{r-1}\!\!\sum_{a \in S\cap V_t} \sum_{b\in V_t} 
				\Dfrac{|\theta_a-\theta_b|_{2\pi}^2}{2} \\
				&\ge m^{r-1} |S\cap V_t| (m-m/r) \Dfrac{\rho_0^2}{32r^2}.
			\end{aligned}
		\end{equation}
		If $|S\cap V_t| \cdot \Sigma_1 \ge m\Sigma_2$ then from \eqref{ineq:sigma12} we obtain 
		$\Sigma_1 = \Omega\left(m^{r} \rho_0^2 /r^2\right) $. Using \eqref{F-2pi-bound}, we find that 
		\[
		|F(\thetavec)| \le \exp\left(-\Omega\left(\varLambda m^{r }   \rho_0^2/r^2\right)\right).
		\]
		By reasoning as in Lemma~\ref{L:B1}, we show that the contribution of such $\thetavec$ is negligible.
		
		Next, we consider the case when $|S\cap V_t| \cdot \Sigma_1 < m \Sigma_2$. 
		Using \eqref{ineq:sigma12} and Lemma~\ref{l:bound-factor} for each  factor in $F(\thetavec)$ corresponding to $\Sigma_2$, we get that 
		\[
		|F(\thetavec)| \le \exp\Bigl(-\Omega(\varLambda|S\cap V_t|m^{r-1} \rho_0^2/r^2  )\Bigr) |\hat{F}(\hat{\thetavec})|,
		\]
		where $\hat{\thetavec}$ consists of the components of $\theta_j$ with $j\in [n]\setminus S$ and  
		\[   
		|\hat{F}(\hat{\thetavec})|=
		\prod_{\substack{e \in  \rsets \\ e \cap S = \emptyset}}
		\bigl|1 + \lambda (e^{ i\sum_{j\in e}\theta_j}-1)\bigr|.
		\]
		Recalling  from \eqref{def:rho0} the definition of  $\rho_0$ and that $|S\cap V_t| = |S|/r$, we find that
		\[
		\varLambda\,\card{S\cap V_t}\,m^{r-1} \rho_0^2/r^2  \ge \card{S} r^2 \log^2 n .
		\]
		Note also that the definition of $\calB_{\kvec,S}$ and \eqref{eq:ourcase}  implies $\norm{\hat{\thetavec}}_\infty \le \rho_0$.
		Integrating over the components corresponding to $S$, together with $(2\pi)^{|S|} =e^{o(|S| r^2 \log^2 n)}$ leads to
		\begin{equation}\label{eq:int-inthat}
			\int_{\calB_{\kvec,S}} \abs{F(\thetavec)}\, d\thetavec  
			\le e^{-\Omega(|S| r^2 \log^2 n)}\int_{\hat{\calB} \cap U_{r\hat{m}}(\rho_0)} \abs{\hat{F}(\hat{\thetavec})}\, d\hat{\thetavec},
		\end{equation}
		where $\hat{\calB}$ is defined as  $\calB$ for reduced parameter 
		\begin{equation}\label{eq:hatm}
			\hat{m} := m - |S|/r \ge 2m/3. 
		\end{equation}

		Finally, using Lemma~\ref{l:abs_inside} to estimate the integral in the right-hand side of \eqref{eq:int-inthat}, we obtain 
		\begin{align*}
			\int_{\calB_{\kvec,S}} |F(\thetavec)|\,d\thetavec   &\le
			e^{-\Omega(|S| r^2 \log^2 n)}  
			{\hat{m}}^{(r-1)/2}  \left(\frac{2\pi }{\varLambda \hat{m}^{r-1}}\right)^{(n-|S|-r+1)/2} 
			\exp\left( O\left(\frac{r^2}{\varLambda \hat{m}^{r-2}}+(r \hat{m})^{-2/5} \right)\right)
			\\
			&\le  e^{-\Omega(|S| r^2 \log^2 n)}     m^{(r-1)/2}\left(\frac{2\pi }{\varLambda m^{r-1}}\right)^{(n-r+1)/2}.
		\end{align*}
		To derive that last inequality, we observe that 
		\begin{align*}
			\hat{m}^{(r-1)/2} \left(\frac{2\pi }{\varLambda \hat{m}^{r-1}}\right)^{(n-|S|-r+1)/2} 
			=  
			\exp\Bigl(O\left(|S|\log (\varLambda \hat{m}^{r-1}) + |S|r\right)\Bigr)   m^{(r-1)/2}\left(\frac{2\pi }{\varLambda m^{r-1}}\right)^{(n-r+1)/2}
		\end{align*}
		and
		\[
		r^2 \log^2 n \gg  \frac{r^2}{\varLambda \hat{m}^{r-2}} + (r \hat{m})^{-2/5} +  \log (\varLambda \hat{m}^{r-1}) + r,
		\]
		which is straightforward by \eqref{eq:hatm} and the assumptions.
	\end{proof}
	
	Now we are ready to  show that the region 
	$\calB \setminus U_{n}(\rho_0)$ has a  negligible contribution.
	\begin{lemma}\label{L:outside}
		If $\varLambda=\Omega(r^{16}m^{2-r})$, we have 
		\[
		\int_{\calB \setminus U_n(\rho_0)} |F(\thetavec)| \,d\thetavec\le 
		e^{-\omega(r  \log n)}
		m^{(r-1)/2}  \left(\frac{2\pi }{\varLambda m^{r-1}}\right)^{(n-r+1)/2}. 
		\]
	\end{lemma}
	\begin{proof}
		Combining Lemma~\ref{L:B1} and Lemma~\ref{L:B2}, we obtain 
		\[
		\int_{ \calB_1\cup\calB_2} |F(\thetavec)|\,d\thetavec\le 
		m^{(r-1)/2}  \left(\frac{2\pi }{\varLambda m^{r-1}}\right)^{(n-r+1)/2} e^{-\omega(r n \log n)}. 
		\]
		By the definitions of $\calB_1$ and $\calB_2$, we find that if $\thetavec \in \calB \setminus (\calB_1\cup\calB_2)  $ then there is $\kvec=(k_1,\ldots,k_r) \in V_1\times \cdots\times V_r$ such that 
		\[  |\theta_{k_1} + \cdots +\theta_{k_r}|_{2\pi} \le \dfrac{\rho_0}{3} 
		\quad \text{and} \quad 
		\Bigl|\{\theta_j\}_{j \in V_t}  \cap \bigl[\theta_{k_t} \pm \dfrac{\rho_0}{4r}\bigr]_{2\pi}\Bigr|\ge m - m/r \text{ for all $t \in [r]$.}
		\]
		Next,  if $|\theta_j - \theta_{k_t}|_{2\pi} \le \dfrac{\rho_0}{2r}$ for all $t\in [r]$ and $j \in V_t$ then, recalling $\theta_{2m}=\cdots= \theta_{rm}=0$, we have 
		\begin{align*}
			|\theta_{k_t}| &= |\theta_{k_t} - \theta_{tm}| \le \dfrac{\rho_0}{2r}, \; \text{ for $t=2,\ldots,r$,}\quad \text{and}\quad
			|\theta_{k_1}| \le \dfrac{\rho_0}{3} + \sum_{t =2}^r |\theta_{k_t}| \le \dfrac{(5r-3)\rho_0}{6r}.
		\end{align*}
		Together with the fact that $\rho_0/(2r)+\rho_0/(2r)\le \rho_0$ and $(5r-3)\rho_0/(6r)+\rho_0/(2r)\le \rho_0$
		this implies that all such $\thetavec$ are in $U_n(\rho_0)$. Therefore, $\calB \setminus(\calB_1 \cup \calB_2 \cup U_n(\rho_0))$ is covered by the sets $\calB_{\kvec,S}$ considered in Lemma~\ref{L:BS}, where $S\neq \emptyset$ and $|S|\le m$. 
		Summing over $j=|S|/r$ and 
		allowing $m^r$ for the choices of $\kvec$ and $m^{rj}$ for the choices of $S$, we get 
		\[
		\int_{\calB \setminus(\calB_1 \cup \calB_2 \cup U_n(\rho_0))} |F(\thetavec)|\,d\thetavec \le \sum_{j = 1}^{m/r} m^{(j+1)r}
		e^{-\omega( jr  \log n)} m^{(r-1)/2}  \left(\frac{2\pi }{\varLambda m^{r-1}}\right)^{(n-r+1)/2},
		\]
		where the $\omega( jr  \log n)$ in the exponent are uniform in $j$. Thus
		\[
		\int_{\calB \setminus(\calB_1 \cup \calB_2 \cup U_n(\rho_0))} |F(\thetavec)|\,d\thetavec\le e^{-\omega( r  \log n)} m^{(r-1)/2}  \left(\frac{2\pi }{\varLambda m^{r-1}}\right)^{(n-r+1)/2},
		\]
		as claimed.
	\end{proof}
	
	
	\subsection{Inside the main region}\label{S:inside}
	
	Turning to the integral inside $U_n(\rho_0)$, we see from the next lemma that the main term in the expansion of $\log F(\thetavec)$ around the origin is $-\thetavec\trans A\thetavec$, where  $A$ is an $n\times n$ matrix defined by 
	\begin{equation}\label{def:A}
		A:=\dfrac{1}{2}\varLambda m^{r-1}(I-B/m+J_{n}/m).
	\end{equation}
	Here $J_{k}$ is the $k\times k$ all-ones matrix, and  $B$ is the block diagonal matrix consisting of $r$ blocks of $J_{m}$.  
	Let  $a_i(\lambda)$ denote the coefficients  of Taylor's expansion of 
	$\log (1+\lambda(e^{ix}-1))$ around the origin.
	In particular, we have
	\begin{align*} a_ 1(\lambda) &= i\lambda,\qquad
		a_ 2(\lambda)= -\dfrac{1}{2} \varLambda,\qquad
		a_ 3(\lambda)= -\dfrac{i}{6}\varLambda(1-2\lambda),\\
		a_ 4(\lambda)&= \dfrac{1}{24}\varLambda(1-6\varLambda),  \qquad a_5(\lambda)= O(\varLambda). 
	\end{align*}
	\begin{lemma}\label{L:Taylor}
		For $r\ge 3$, $\varLambda=\Omega(r^{16}m^{2-r})$, 
		and constant~$\gamma>0$,
		\[
		\log F(\thetavec)  
		=  -   \thetavec\trans A \thetavec
		+ \sum_{p=3}^4\,
		\sum_{e \in\rsets}\!\!  a_ p(\lambda)\, \biggl(\, \sum_{j\in e} \theta_j\biggr)^{\!p}
		+  O(n^{-1/2}\log^5 n)
		\]
		uniformly for $\thetavec\in U_n(\gamma\rho_0)$.
	\end{lemma}
	\begin{proof}
		Observing that for all $e\in \rsets$ we have $|\sum_{j\in e}\theta_j|= O(r\rho_0)$ for $\thetavec\in U_n(O(\rho_0))$ and 
		using Taylor's theorem, we obtain
		\[
		\log \bigl(1 + \lambda (e^{ i\sum_{j\in e}\theta_j}-1)\bigr) = \sum_{p=1}^4 \, a_p(\lambda)  \biggl(\, \sum_{j\in e} \theta_j\biggr)^{\!p}
		+ O\left( \varLambda (r \rho_0)^5\right).
		\]
		Summing over all $e\in \rsets$, we find that 
		the linear term of $\log F(\thetavec)$ (which includes terms from the denominator of $F(\thetavec)$), is
		\[  i \, \sum_{j\in [n]} \theta_j\, (\, m^{r-1}\lambda - d ),  \]
		which is zero as $d=\lambda m^{r-1}$. 
		Furthermore, in $-\thetavec\trans A \thetavec$ the coefficient of $\theta_j^2$ is $-\varLambda m^{r-1}/2$ for any $j$, while the coefficient of $\theta_j\theta_k$ is $-\varLambda m^{r-2}$ if $j,k$ belong to different partition classes, and 0 if $j,k$ are within the same partition class, matching the quadratic term of $\log F(\thetavec)$.
		It remains to observe that the combined error term
		\[
		O(\varLambda m^{r}(r \rho_0)^5) = O\left(
		\frac{m^{-3r/2+5/2} r^{35/2} \log^5 n}{\varLambda^{3/2}}
		\right) = O(n^{-1/2}\log^5 n),
		\]
		using $\varLambda=\Omega(r^{16}m^{2-r})$ and Lemma~\ref{L:errors}.
	\end{proof}
	
	Note that $A$ is singular of nullity $r-1$.
	We evaluate $\int_{\calB\cap U_n(\rho_0)} F(\thetavec)\,d\thetavec$ using 
	the methods from~\cite{Mother} to first raise it to an integral
	over a domain of full dimension, and then to evaluate the
	resulting integral.
	
	The following matrices will play a relevant role in our proof.
	\begin{equation}\label{def:WT}
		W := \biggl( \Dfrac{\varLambda m^{r-3}}{2r}\biggr)^{\!1/2} (rB-J_n) \quad \text{and}
		\quad T := \left(\Dfrac{2}{\varLambda m^{r-1}}\right)^{\!1/2} \left(I-\Dfrac{\sqrt{r}-1}{\sqrt{r}m}B\right).
	\end{equation}
	
	\begin{lemma}\label{L:linear}
		The  following identities and bounds hold.
		\begin{itemize}
			\item[(a)] $A+W\trans W$ has $r$ eigenvalues equal to
			$\frac12 \varLambda rm^{r-1}$ and $n-r$ eigenvalues
			equal to $\frac12 \varLambda m^{r-1}$.
			\item[(b)] $\abs{A+W\trans W}
			=\left(\dfrac{\varLambda m^{r-1}}{2}\right)^{n}r^r$.
			\item[(c)]  $T\trans (A+W\trans W)T=I$.
			\item[(d)] $(A + W\trans W)^{-1} 
			= \dfrac{2}{\varLambda m^{r-1}}\left(I-\dfrac{r-1}{n}B\right)$,
			\item[(e)] $\norm{T}_\infty = 
			\norm{T}_1 \le 
			\dfrac{3}{\left(\varLambda m^{r-1}\right)^{1/2}}$.
			\item[(f)] 
			$\norm{T^{-1}}_\infty 
			= 2^{-1/2}\bigl( \varLambda r m^{r-1}\bigr)^{1/2}$.
		\end{itemize}
	\end{lemma}
	
	\begin{proof}
		Using $B^2=mB$, $J_nB=BJ_n=mJ_n$ and $J_n^2=rmJ_n$, it is routine to
		verify that
		\begin{align*}
			A+W\trans W &= \dfrac12\varLambda m^{r-1}
			\Bigl(I + \Dfrac{r-1}{m}B\Bigr), \quad\text{and} \\
			T^{-1} 
			&=  2^{-1/2}\varLambda^{1/2}m^{(r-1)/2}\Bigl(I+\Dfrac{\sqrt{r}-1}{m}B\Bigr).
		\end{align*}
		The eigenspaces of $B$ are those that are constant on each class (dimension $r$)
		and those that sum to 0 on each class (dimension $n-r$).
		This proves (a), and (b) immediately follows.
		Parts (c) and (d) are proved by direct multiplication, while
		(e) and (f) follow from the explicit forms of
		$T$ and $T^{-1}$.
	\end{proof}

	The kernel of matrix $A$ from \eqref{def:A} consists of  the set of all vectors which sum to~0 and are constant within each class. 
	Note that  $\ker A$ is   the subspace spanned by the vectors
	$ \vvec_2,\ldots,\vvec_r$ defined by 
	\begin{equation}\label{eq:vdef}
		\vvec_j=\sum_{\ell \in V_j}\evec_{\ell}-\sum_{\ell\in V_1}\evec_{\ell}, 
		\qquad \text{for $2\le j\le r$,}
	\end{equation}
	where $\evec_{\ell} \in \Reals^n$ is the standard basis vector whose
	$\ell$-th coordinate is 1 and all others are~0.
	Let  $Q$ denote a  projection operator  
	into the space  $\{\xvec\in \Reals^n \st x_{2m}= \cdots= x_{rm} =0\}$
	defined by 
	\[
	Q := I - \sum_{j=2}^r \vvec_{j} \evec_{j m}\trans.
	\]
	Note that $U_n(\rho_0)\cap\calB = U_n(\rho_0)\cap Q(\Reals^n)$.

	
	
	We will need Lemma~4.6 from \cite{Mother}. 
	This lemma allows us to raise the integral to a domain of full dimension. In particular, it shows that when integrating a function $G(\yvec)$ associated with a singular matrix in its expansion, $G(\yvec)$ can be replaced by a function $G(Q\xvec)e^{-\xvec\trans W\trans W \xvec}$ associated with a nonsingular matrix at the cost of a multiplicative factor and a change in the range. In addition, the lemma provides convenient bounds for this new range.
	
	\begin{lemma}[Isaev and McKay \cite {Mother}]\label{LemmaQW}
		Let $Q,W:\mathbb{R}^n\rightarrow \mathbb{R}^n$ be linear operators 
		such that $\ker Q \cap \ker W = \{\boldsymbol{0}\}$ and $\operatorname{span}(\ker Q, \ker W) = \Reals^n$. 
		Let $\nperp$ denote the dimension of\/ $\ker Q$.
		Suppose $\varOmega \subseteq \Reals^n$ and
		$G:\varOmega\cap Q(\Reals^n) \to\Complexes$.
		For any $\tau>0$, define
		\[
		\bar\varOmega := \bigl\lbrace \xvec\in\Reals^n \st
		Q\xvec\in \varOmega \text{~and~}
		W\xvec\in U_n(\tau) \bigr\rbrace.
		\]
		Then, if the integrals exist,
		\[
		\int_{\varOmega \cap Q(\Reals^n)} G(\yvec)\,d\yvec
		= (1 - K)^{-1}\,\pi^{-\nperp/2} \,\Abs{Q\trans Q + W\trans W}^{1/2}
		\int_{\bar\varOmega} G(Q\xvec)\, e^{-\xvec\trans W\trans W \xvec}
		\,d\xvec,
		\]
		where 
		\[
		0\le K < \min(1,n e^{-\tau^2}).
		\] 
		Moreover, if  $U_n(\rho_1) \subseteq \varOmega \subseteq  U_n(\rho_2)$ for some $\rho_2 \ge \rho_1 >0$ then
		\[
		U_n\biggl(\min\biggl(\frac{\rho_1}{\norm{Q}_\infty},
		\frac{\tau}{\norm{W}_\infty} \biggr)\biggr)
		\subseteq \bar\varOmega \subseteq
		U_n\bigl( \norm{P}_\infty\, \rho_2 +  \norm{R}_\infty\, \tau \bigr)
		\]	
		for any linear operators $P,R : \Reals^n \rightarrow \Reals^n$ such that $PQ + RW$ is equal to the identity operator on $\Reals^n$.  
	\end{lemma}
	
	Applying this lemma for $F(\thetavec)$ and $|F(\thetavec)|$ gives the following results.
	
	\begin{lemma}\label{L:QQWW}
		There is some region $\bar\varOmega$ with
		$U_n\bigl(\dfrac{2\rho_0}{3r}\bigr)\subseteq \bar\varOmega
		\subseteq U_n(4\rho_0)$
		such that for $G(\xvec)=F(\xvec)$ and $G(\xvec)=|F(\xvec)|$ we have
		\begin{align*}
			\int_{U_n(\rho_0)\cap Q(\Reals^n)} \!\!G(\xvec)\,d\xvec 
			= (1+e^{-\omega(\log n)}) \pi^{-(r-1)/2}
			\Abs{Q\trans Q+W\trans W}^{1/2}
			\int_{\bar\varOmega} G(\xvec) 
			e^{-\xvec\trans W\trans W\xvec}\,d\xvec.
		\end{align*}
	\end{lemma}
	
	\begin{proof}
		We can apply Lemma~\ref{LemmaQW} with $\varOmega=U_n(\rho_0)$,
		$\rho_1=\rho_2=\rho_0$, $\tau=r^2\log n$ and the matrices
		$Q,W$ we have defined.
		In addition, define the matrices
		\[
		P:=I-\Dfrac{1}{r}\left(\Dfrac{2r}{\varLambda m^{r-1}}\right)^{\!1/2}W 
		\quad \text{and} 
		\quad R:=\Dfrac{1}{r}\left(\Dfrac{2r}{\varLambda m^{r-1}}\right)^{\!1/2}I.
		\]
		Note that for $2\le j \le r$ and the vectors $\vvec_j$
		defined in~\eqref{eq:vdef} we have
		\[  
		W \vvec_{j}=\left(\Dfrac{\varLambda m^{r-1}}{2r}\right)^{1/2} r v_j,  \]
		implying $W (I-Q)=\left(\dfrac{\varLambda m^{r-1}}{2r}\right)^{1/2} r (I-Q)$.
		Therefore
		\[
		PQ+RW=Q+\frac{1}{r}\left(\frac{2r}{\varLambda m^{r-1}}\right)^{1/2}W(I-Q)=I.
		\]
		
		We have $\nperp=r-1$, and
		\eqref{F-symmetry} implies $G(Q\xvec)=G(\xvec)$ for all~$\xvec$.
		From their explicit forms we have
		$\norm{Q}_\infty=r$, 
		$\norm{R}_\infty= 2^{1/2}\varLambda^{-1/2}r^{-1/2}m^{-(r-1)/2}$,
		$\norm{W}_\infty = (r-1)r^{-1/2}(2\varLambda m^{r-1})^{1/2}$
		and $\norm{P}_\infty\le 3$ follows from $\norm{W}_\infty$.
		Therefore
		\[
		\min\biggl(\frac{\rho_1}{\norm{Q}_\infty},
		\frac{\tau}{\norm{W}_\infty}\biggr)
		= \min\biggl(\Dfrac {\rho_0}{r},\Dfrac{\rho_0\sqrt2}{2r-2}\biggr)
		\ge \Dfrac{2\rho_0}{3r},
		\]
		and
		\[
		\norm{P}_\infty\, \rho_2 +  \norm{R}_\infty\, \tau
		\le 3\rho_0 + \sqrt{2}\,r^{-1}\rho_0 \le 4\rho_0.
		\]
	\end{proof}
	
	Again we turn to the theory in \cite{Mother} to estimate the integrals of $F(\xvec)e^{-\xvec\trans W\trans W\xvec}$ and $|F(\xvec)|e^{-\xvec\trans W\trans W\xvec}$. We start with the following definitions.
	
	Recall that for a complex random variable $Z$, the \textit{variance} is defined by
	\[ \Var Z = \E \,\abs{Z - \E Z}^2 = \Var \Re Z + \Var \Im Z,\]
	while the \textit{pseudovariance} is
	\[ \V Z = \E (Z - \E Z)^2 = \Var \Re Z - \Var \Im Z + 2i\, \Cov(\Re Z, \Im Z).\]
	In addition, for a domain $\varOmega\subset \Reals^n$ and a twice continuously differentiable function $g:\varOmega\to \Complexes$, let
	\[H(g,\varOmega):=(h_{jk})\quad \mbox{where}\quad 
	h_{jk}=\sup_{\xvec\in\varOmega}\, \left| \Dfrac{\partial^2 g}{\partial x_j\, \partial x_k}(\xvec) \right|.\]
	
	The following theorem is a simplified version of Theorem~4.4 in \cite{Mother}.
	It estimates the integral of a function of the form $e^{-\xvec\trans \tilde A\xvec + f(\xvec)+h(\xvec)}$ by comparing it to the integral of a truncated multivariate Gaussian distribution.
	In order to achieve an accurate estimate, the range of the integral and $f(x)$ need to satisfy a stringent set of conditions, while $h(x)$ needs to be sufficiently small.
	\begin{theorem}[Isaev and McKay \cite{Mother}]\label{gauss4pt}
		Let $c_1,c_2,\rho_1,\rho_2,\phi_1,\phi_2$ be
		nonnegative real constants.
		Let $\medtilde A$ be an $n\times n$ positive definite symmetric real matrix
		and let $T$ be a real matrix such that $T\trans \medtilde AT=I$.
		Let $\varOmega$ be a measurable set such that
		$T(U_n(\rho_1))\subseteq \varOmega\subseteq T(U_n(\rho_2))$,
		and let 
		$f: \Reals^n\to\Complexes$ and $h:\varOmega\to\Complexes$ 
		be measurable functions.
		Assume the following conditions:
		\begin{itemize}
			\item[(a)] $\log n\le\rho_1\le\rho_2$.
			\item[(b)] For $\xvec\in T(U_n(\rho_1))$
			and $1\le j\le n$,
			\begin{align*}
				2\rho_1\,\norm{T}_1\,\left|\dfrac{\partial f}{\partial x_j}(\xvec)\right|
				\le \phi_1 n^{-1/3}\le\dfrac23, \qquad\text{and}\\
				4\rho_1^2\,\norm{T}_1\,\norm{T}_\infty\,
				\norm{H(f,T(U_n(\rho_1)))}_\infty
				\le \phi_1 n^{-1/3}.
			\end{align*}
			\item[(c)] 
			For $\xvec\in T(U_n(\rho_2))$ and $1\le j\le n$,
			\[ 2\rho_2\,\norm{T}_1\,\left|\dfrac{\partial\, \Re f}{\partial x_j}(\xvec)\right|\le 
			(2\phi_2)^{3/2} n^{-1/2}.
			\]
			\item[(d)] $\abs{f(\xvec)} \le n^{c_1} 
			e^{c_2\xvec\trans \tilde A\xvec/n}$ for $\xvec\in\Reals^n$.
		\end{itemize}
		Let $\Xvec$ be a random variable with the normal density
		$\pi^{-n/2} \abs{\medtilde A}^{1/2} e^{-\xvec\trans\!\medtilde A\xvec}$.
		Then, provided $\V f(\Xvec)$ 
		and $\V\,\Re f(\Xvec)$ are finite
		and $h$ is bounded in~$\varOmega$,
		\[
		\int_\varOmega e^{-\xvec\trans\tilde A\xvec + f(\xvec)+h(\xvec)}\,d\xvec
		= (1+K)\, \pi^{n/2}\abs{\medtilde A}^{-1/2} e^{\E f(\Xvec  )+\frac12\V f(\Xvec  )},
		\]     
		where, for large enough $n$, $K=K(n)$ satisfies
		\[
		\abs{K} \le e^{\frac12\Var\Im f(\Xvec  )}\,\bigl( e^{\phi_1^3+e^{-\rho_1^2/2}}
		+ 2e^{\phi_2^3+e^{-\rho_1^2/2}}-3
		+ \sup_{\xvec\in\varOmega}\,\abs{e^{h(\xvec)}-1}
		\bigr).
		\]
	\end{theorem}
	
	We will need the following theorem of Isserlis
	(also known as Wick's formula),
	see for example~\cite[Theorem~1.1]{MNBO}.
	
	\begin{theorem}\label{isserlis}
		Let $D$ be a positive-definite real symmetric matrix of
		order~$n$ and
		let  $\Xvec=(X_1,\ldots,X_n)$ be a random vector with
		Gaussian density $\pi^{-n/2}\abs{D}^{1/2} e^{-\xvec\trans D\xvec}$.
		Let $\varSigma=(\sigma_{jk})=(2D)^{-1}$ be the corresponding
		covariance matrix.
		Consider a product $Z=X_{j_1}X_{j_2}\cdots X_{j_k}$, where the
		subscripts do not need to be distinct.  If $k$ is odd, then
		$\E Z=0$.  If $k$ is even, then
		\[
		\E Z = \sum_{\{(i_1,i_2),(i_3,i_4),\ldots,(i_{k{-}1},i_k)\}}
		\sigma_{j_{i_1}j_{i_2}}\cdots\sigma_{j_{i_{k{-}1}}j_{i_k}},
		\]
		where the sum is over all pairings of $\{1,\ldots,k\}$.
		The sum has $(k-1)(k-3)\times\cdots\times3\times 1$ terms.
	\end{theorem}

	\begin{lemma}\label{L:inbox}
		If $\varLambda=\Omega(r^{16}m^{2-r})$ and $\bar\varOmega$
		is the domain provided by Lemma~\ref{L:QQWW},
		\[
		\int_{\bar\varOmega}
		F(\thetavec)e^{-\thetavec\trans W\trans W\thetavec} \, d\thetavec 
		= \Dfrac{\pi^{n/2}}{|A+W\trans W|^{1/2}}\, 
		\exp\Bigl(- \Dfrac{r}{12\varLambda m^{r-2}} + O(n^{-2/5})\Bigr).
		\]
	\end{lemma}
	
	\begin{proof}
		We will apply Theorem~\ref{gauss4pt} with $\medtilde A=A+W\trans W$, $\rho_1=\log n$, $\rho_2=3r^3\log n$,
		$\varOmega=\bar\varOmega$ and $T$ as in~\eqref{def:WT}.
		
		We first verify the condition
		$T(U_n(\rho_1))\subseteq \bar\varOmega\subseteq T(U_n(\rho_2))$. Using the bounds for $\norm{T}_{\infty}$ and $\norm{T^{-1}}_{\infty}$ established in Lemma~\ref{L:linear} gives
		\begin{align}
			T(U_n(\rho_1)) &\subseteq U_n(\norm{T}_\infty\log n)
			\subseteq U_n(3\varLambda^{-1/2}m^{-(r-1)/2}\log n)
			\subseteq U_n\Bigl(\Dfrac{2\rho_0}{3r}\Bigr),
			\label{box1}\\
			U_n(4\rho_0) &\subseteq T(U_n(4\norm{T^{-1}}_\infty\rho_0))
			= T(U_n(2^{3/2}r^3\log n)) \subseteq T(U_n(\rho_2)).\notag
		\end{align}
		
		By Lemma~\ref{L:Taylor}, we can take
		$f(\xvec)=if_3(\xvec)+f_4(\xvec)$, where
		\[
		f_3(\xvec) = -\dfrac{1}{6}\,\varLambda(1-2\lambda)
		\sum_{e\in\rsets} \biggl(\sum_{j\in e}x_j\biggr)^{\!3},
		\quad
		f_4(\xvec) = \dfrac{1}{24}\,\varLambda(1-6\varLambda)
		\sum_{e\in\rsets}\biggl(\sum_{j\in e}x_j\biggr)^{\!4},
		\]
		and $h(\xvec)=O(n^{-1/2}\log^5 n)$.
		
		By~\eqref{box1}, $\xvec\in T(U_n(\rho_1))$ implies
		$\norm{\xvec}_\infty\le 3\varLambda^{-1/2}m^{-(r-1)/2}\log n$.
		Consequently, for any~$j$,
		since $m^{r-1}$ elements of $\rsets$ are incident with
		vertex~$j$
		\begin{align*}
			2\rho_1\norm{T}_1\Bigl|\dfrac{\partial f}{\partial x_j}(\xvec)\Bigr|
			&=
			O(1)\log n\,\norm{T}_1 \varLambda m^{r-1}\bigl((r\norm{\xvec}_\infty)^2
			+ (r\norm{\xvec}_\infty)^3\bigr) \\
			&= O(\varLambda^{-1/2}r^2 m^{-(r-1)/2}\log^3 n)
			= O(n^{-1/2}\log^3 n),
		\end{align*}
		where in the final step we used $\varLambda=\Omega(r^{16}m^{2-r})$
		and Lemma~\ref{L:errors}.
		Similarly,\\
		$\norm{H(f,T(U_n(\rho_1)))}_\infty
		= O(\varLambda^{1/2}r^2m^{(r-1)/2}\log n)$,
		which implies
		\[
		4\rho_1^2\norm{T}_1\norm{T}_\infty\norm{H(f,T(U_n(\rho_1)))}_\infty
		= O(\varLambda^{-1/2}r^2m^{-(r-1)/2}\log^3 n)
		= O(n^{-1/2}\log^3 n),
		\]
		using $\varLambda=\Omega(r^{16}m^{2-r})$.
		Consequently, part (b) of Theorem~\ref{gauss4pt} is satisfied
		with $\phi_1=n^{-1/7}$ if $n$ is sufficiently large.

		If $\xvec\in T(U_n(\rho_2))$ then
		$\norm{\xvec}_\infty \le \norm{T}_\infty\rho_2
		\le 9\varLambda^{-1/2}r^3m^{-(r-1)/2}\log n$.
		Noting that $\Re f(\xvec)=f_4(\xvec)$, we calculate
		\begin{align*}
			2\rho_2\norm{T}_1\Bigl|\dfrac{\partial f_4}{\partial x_j}(\xvec)\Bigr|
			&=
			O(1)r^3\log n\norm{T}_1 \varLambda m^{r-1}(r\norm{\xvec}_\infty)^3 \\
			&= O(\varLambda^{-1}r^{15} m^{1-r}\log^4 n)
			= O(n^{-1}\log^4 n),
		\end{align*}
		using $\varLambda=\Omega(r^{16}m^{2-r})$, so part (c) of Theorem~\ref{gauss4pt} is satisfied
		with $\phi_2=n^{-1/4}$ if $n$ is sufficiently large.

		We next check condition (d) of Theorem~\ref{gauss4pt}.
		Define $\ell(\xvec):=\dfrac1n\, \xvec\trans (A+W\trans W)\xvec$.
		Since the least eigenvalue of $A+W\trans W$ is
		$\frac12\varLambda m^{r-1}$, we have
		\[
		\norm{\xvec}_\infty^2 \le \norm{\xvec}_2^2
		= O(\varLambda^{-1}rm^{2-r}\ell(\xvec)).
		\]
		Consequently,
		\begin{align*}
			f(\xvec) &= O(\varLambda r^3m^r\norm{\xvec}_\infty^3
			+ \varLambda r^4 m^r\norm{\xvec}_\infty^4) \\
			&= O(\varLambda r^2m^r\norm{\xvec}_\infty^2
			+ \varLambda r^4 m^r\norm{\xvec}_\infty^4) \\
			&= O(r^3m^2)\bigl(\ell(\xvec)+\varLambda^{-1}r^3m^{2-r}
			\ell(\xvec)^2\bigr),
		\end{align*}
		where the second line is because $\abs{x}^3\le x^2+x^4$ for all~$x$.
		By $\varLambda=\Omega(r^{16}m^{2-r})$,
		$\varLambda^{-1}r^3m^{2-r}=O(1)$.
		Therefore,
		\[
		f(\xvec) = O(r^3m^2)\bigl(\ell(\xvec)+\ell(\xvec)^2\bigr)
		= O(n^5) e^{\ell(\xvec)},
		\]
		as required.
		
		Now let $\Xvec=(X_1,\ldots,X_n)$ be a Gaussian random vector with
		density proportional to $e^{-\xvec\trans(A+W\trans W)\xvec}$.
		To complete our calculation, we need the expectation and
		pseudovariance of $f(\Xvec)$ and the variance of~$f_3(\Xvec)$.
		The covariance matrix of $\Xvec$ is $\frac12(A+W\trans W)^{-1}$,
		and so by Lemma~\ref{L:linear},
		\[
		\Cov\left[X_j,X_k\right]=
		\begin{cases}
			\dfrac{1}{\varLambda m^{r-1}}\Bigl(1-\dfrac{r-1}{rm}\Bigr),  & \mbox{if } j=k; \\[0.3ex]
			-\dfrac{r-1}{\varLambda r m^r}, & \mbox{if $j\ne k$ in the same class;}\\[0.3ex]
			0, &\mbox{otherwise.}
		\end{cases}
		\]
		For $e\in\rsets$ define $X_e=\sum_{j\in e} X_j$ and, for $e,e'\in\rsets$ define
		\[
		\sigma(e,e'):=\Cov[X_e,X_{e'}].
		\]
		Since covariance is additive and $ e$ and $e'$ contain one vertex from each class, we have
		\[
		\sigma(e,e') = \varsigma(\card{e\cap e'}), \text{~~where~~}
		\varsigma(k) := \frac{1}{\varLambda m^{r-1}}\Bigl(k -\Dfrac{r-1}{m} \Bigr).
		\]
		Since $f_3$ is an odd function, we have $\E f_3(\Xvec)=0$.
		Theorem~\ref{isserlis} implies that 
		$\E X_e^4 = 3\sigma(e ,e)^2$, so
		\[
		\E f(\Xvec) = \Dfrac{(1-6\varLambda)(mr-r+1)^2}{8\varLambda m^r} = 
		\Dfrac{r^2}{8\varLambda m^{r-2}} + O(n^{-1}),
		\]
		using $\varLambda=\Omega(r^{16}m^{2-r})$ and Lemma~\ref{L:errors}.
		
		By Theorem~\ref{isserlis}, for any pair $e,e'\in\rsets$,
		\[
		\E [X_e^3 X_{e'}^3] 
		= 9\, \sigma(e ,e)\, \sigma(e',e')\, \sigma(e,e') +6 \, \sigma(e,e')^3 .
		\]
		The number of pairs with $\card{e\cap e'}=k$ is $m^r\binom rk(m-1)^{r-k}$.
		Summing over $e,e'\in\rsets$, and applying
		$\varLambda=\Omega(r^{16}m^{2-r})$, we obtain
		\begin{align}
			\Var f_3(\Xvec) &= \Dfrac{\varLambda^2(1-2\lambda)^2
				m^r(m-1)^r}{36} \sum_{k=0}^r \binom rk
			(m-1)^{-k}\bigl(9\varsigma(r)^2\varsigma(k)+6\varsigma(k)^3\bigr) \notag\\
			&=
			\Dfrac{(1-4\varLambda) r(3r+2)m}{12\varLambda m^{r-1}}
			\Bigl(1 - \Dfrac{6(r-1)}{(3r+2)m} + \Dfrac{(r-1)(3r-5)}{r(3r+2)m^2}\Bigr) \notag\\
			&= \Dfrac{r(3r+2)}{12\varLambda m^{r-2}} + O(n^{-1}) \label{eq:var3} \\
			&= O(1). \label{eq:VarIm}
		\end{align}
		
		Similarly,
		\begin{align*}
			\Cov [X_e^4,X_{e'}^4] &=
			\E (X_e^4 X_{e'}^4) - (\E X_e^4)(\E X_{e'}^4) \\
			&= 72\sigma(e ,e)\sigma(e',e')\sigma(e,e')^2
			+ 24\sigma(e,e')^4.
		\end{align*}
		Summing over $e,e'\in\rsets$ with $\varLambda=\Omega(r^{16}m^{2-r})$ gives
		\begin{equation}\label{eq:var4}
			\Var(f_4(\Xvec)) = O(n^{-1}).
		\end{equation}
		Finally, since $\E f_3(\Xvec)=0$ and all the terms of $f_3(\Xvec)f_4(\Xvec)$
		have odd degree, Theorem~\ref{isserlis}
		gives $\Cov(f_3(\Xvec),f_4(\Xvec))= 0$.
		Therefore,
		\eqref{eq:var3} and~\eqref{eq:var4} together imply that
		\[
		\V f(\Xvec) = -\frac{r(3r+2)}{12\varLambda m^{r-2}} + O(n^{-1})
		\]
		and so
		\begin{equation}\label{eq:EhalfV}
			\E f(\Xvec) + \dfrac12 \V f(\Xvec)
			= -\Dfrac{r}{12\varLambda m^{r-2}} + O(n^{-1}).
		\end{equation}
		The lemma now follows from Theorem~\ref{gauss4pt}
		and equations~\eqref{eq:VarIm} and~\eqref{eq:EhalfV}.
	\end{proof}

	The above precise result relies on $d=\lambda m^{r-1}$
	since otherwise the linear term doesn't vanish.
	However, if we are concerned with the integral of
	$\abs{F(\thetavec)}$ we can ignore both the linear and
	cubic terms.  So, with identical proof, the following is true whenever
	$\varLambda=\Omega(r^{16}m^{2-r})$.
	
	\begin{corollary}\label{cor:absinbox}
		For $\bar\varOmega$ as in Lemma~\ref{L:QQWW} we have
		\[
		\int_{\bar\varOmega} e^{-\xvec\trans W\trans W\xvec}\,\abs{F(\xvec)} \, d\xvec 
		= \Dfrac{\pi^{n/2}}{|A+W\trans W|^{1/2}}\, 
		\exp\Bigl(\Dfrac{r^2}{8\varLambda m^{r-2}} + O(n^{-2/5})\Bigr).
		\]
	\end{corollary}

	Our final task is to evaluate the determinant $\abs{Q\trans Q+W\trans W}$.
	Let $\jvec$ be the column vector of all 1s, and recall that
	$\evec_j$ is the $j$-th elementary column vector.
	
	\begin{lemma}\label{L:QWdet}
		$\abs{Q\trans Q+W\trans W} =r^{r}\Bigl(\Dfrac{\varLambda m^r}{2}\Bigr)^{r-1}$.
	\end{lemma}
	
	\begin{proof}
		We proceed in three stages.
		
		\noindent\textit{Claim 1}.
		Consider the block matrix 
		\[ M'= \begin{pmatrix} M_{1,1} & \!M_{1,2}\\
			M_{2,1} & \!M_{2,2}
		\end{pmatrix},
		\]
		where each block has size $m\times m$ and the following form
		\begin{align*}
			M_{1,1}&= I+(r-1)y J_m & M_{1,2}&= (r-1)\jvec\evec_m\trans -(r-1)y J_m\\
			M_{2,1}&= \evec_m \jvec\trans- y J_m & M_{2,2}&= I-\jvec  \evec_m\trans -\evec_m\jvec\trans+ mr \evec_m \evec_m\trans  +y J_m.
		\end{align*}
		Then, we have $\abs{M'}=r^2 m^2 y$.
		
		Subtract from the last row the first $m$ rows and add to it the following $m-1$ rows. Then the last row has the form
		\[ (-rm y,\ldots, -rm y, rm y,\ldots, rm y). \] 
		After dividing this row by $rmy$, it can be used to eliminate the terms of the form $y J_m$.
		
		Finally add to the last row the first $m$ rows and subtract from it the following $m-1$ rows
		to create an upper triangular matrix. Claim~1 follows as every entry in the diagonal is~1 except for the last entry, which is $rm$. 
		
		\noindent\textit{Claim 2}.
		Let the matrix $M''$ have size $m\times m$ with the form 
		\begin{align*}
			M''= I-\jvec\evec_m\trans -\evec_m\jvec\trans+ m \evec_m \evec_m\trans  +y J_m.
		\end{align*}
		Then, we have $\abs{M''}=m^2 y$.
		
		Add the first $m-1$ rows to the last row, which will subsequently have the form
		\[ (my,\ldots,my). \]
		After dividing this row by $my$, it can be used to eliminate the terms of the form $y J_m$.
		Subsequently subtract the first $m-1$ rows from the last row to create an upper triangular matrix.
		This proves Claim~2.
		
		Proceeding now with the proof of the lemma,
		break down the matrix into blocks of size $m\times m$ and consider the following ``block'' operations. Starting with the last row and finishing with the third row subtract from each row the one above it. Then starting from the penultimate column and finishing with the second column add the column on the right to it. These operations lead to an upper triangular block matrix, where the diagonal consists of $r-1$ blocks. The first block has size $2m\times 2m$ and form as in Claim~1 with $y=\varLambda m^{r-1}/(2m)$. The remaining $r-2$ blocks have size $m\times m$ and have the form as in Claim~2 with $y=r\varLambda m^{r-1}/(2m)$.
	\end{proof}
	
Now we can prove Lemmas~\ref{L:dense} and \ref{l:abs_inside}.
	
	\begin{proof}[Proof of Lemma~\ref{l:abs_inside}]
		Lemma~\ref{L:QQWW} and Corollary~\ref{cor:absinbox} implies
		\[
		\int_{\calB\cap U_n(\rho_0)} \!\! |F(\thetavec)|\,d\thetavec
		= \Abs{Q\trans Q+W\trans W}^{1/2}\Dfrac{\pi^{(n-r+1)/2}}{|A+W\trans W|^{1/2}}\, 
		\exp\Bigl(\Dfrac{r^2}{8\varLambda m^{r-2}} + O(n^{-2/5})\Bigr).
		\]
		Using Lemmas~\ref{L:linear}(b) and~\ref{L:QWdet} for the values of the determinants gives the required outcome
		\[
		\int_{\calB\cap U_n(\rho_0)} \!\! |F(\thetavec)|\,d\thetavec
		= m^{(r-1)/2}  \left(\Dfrac{2\pi }{\varLambda m^{r-1}}\right)^{(n-r+1)/2} \exp\Bigl(\Dfrac{r^2}{8\varLambda m^{r-2}} + O(n^{-2/5})\Bigr).
		\]
		
	\end{proof}
	
	\begin{proof}[Proof of Lemma~\ref{L:dense}]
		Lemmas~\ref{L:Qlemma}, \ref{L:outside}, \ref{L:QQWW}, \ref{L:inbox}, \ref{L:QWdet},
		and \ref{L:linear}(b) imply
		\begin{align*}
			H_r(d,m) = \bigl(\lambda^{\lambda} (1-\lambda)^{1-\lambda} \bigr)^{-m^r}\! (2\pi\varLambda)^{(r-n-1)/2} m^{-r(r-1)(m-1)/2}
			\exp\Bigl(- \Dfrac{r}{12\varLambda m^{r-2}} + O(n^{-2/5})\Bigr).
		\end{align*}
		So our remaining task is to verify that $H_r(d,m)$ matches $\hat{H}_r(d,m)$ from \eqref{eq:naive}.
		For $\lambda\in (0,1)$ and integer $\lambda N$, Stirling's expansion
		$k! = (k/e)^k\sqrt{2\pi k}\, e^{1/(12k)+O(k^{-3})}$
		gives
		\[
		\binom{N}{\lambda N}
		= \Dfrac{\bigl(\lambda^\lambda
			(1-\lambda)^{1-\lambda}\bigr)^{-N}}
		{\sqrt{2\pi N\varLambda}}
		\exp\Bigl( -\Dfrac{1-\varLambda}{12\varLambda N}
		+ O(\varLambda^{-3} N^{-3})\Bigr).
		\]
		Now apply this expansion to \eqref{eq:naive} for
		$N=m^r$ and $N=m^{r-1}$, using Lemma~\ref{L:errors}
		and the assumption $\varLambda=\Omega(r^{16}m^{2-r})$.
		This gives the same expression as we obtained above for
		$H_r(d,m)$ 
		with error term $O(n^{-1})$.
		This completes the proof.
	\end{proof}
	
	


\end{document}